\documentclass{article}

\usepackage[utf8]{inputenc}
\usepackage[normalem]{ulem} 	
\usepackage{xparse}         	
\usepackage{comment}

\usepackage{mathtools,amssymb,amsthm,empheq}
\usepackage{bm} 	   	
\usepackage{relsize}   	
\usepackage{mathrsfs}  	
\usepackage{cancel}   	
\usepackage{esint}    	
\usepackage{xcolor} 	
\usepackage{cases}

\usepackage{pgf,tikz}
\usetikzlibrary{cd,arrows,babel,shapes,positioning}

\usepackage{enumerate}

\usepackage{hyperref}
\hypersetup{
	colorlinks,
	citecolor=blue,
	filecolor=blue,
	linkcolor=blue,
	urlcolor=blue
}

\newcommand{\supp}{\mathrm{supp}}

\newcommand{\EPS}{\varepsilon}

\newcommand{\ColorWord}[2]{\color{#1} #2 \color{black} }

\allowdisplaybreaks

\numberwithin{equation}{section}

\theoremstyle{plain}

\newtheorem{thm}[equation]{Theorem}
\newcommand{\refthm}[1]{\emph{\ColorWord{blue}{Theorem} \ref{#1}}}

\newtheorem{lemma}[equation]{Lemma}
\newcommand{\reflemma}[1]{\emph{\ColorWord{blue}{Lemma} \ref{#1}}}

\newtheorem{prop}[equation]{Proposition}
\newcommand{\refprop}[1]{\emph{\ColorWord{blue}{Proposition} \ref{#1}}}

\newtheorem{cor}[equation]{Corollary}
\newcommand{\refcor}[1]{\emph{\ColorWord{blue}{Corollary} \ref{#1}}}

\theoremstyle{definition}

\newtheorem{defin}[equation]{Definition}
\newcommand{\refdef}[1]{\emph{Definition \ref{#1}}}

\theoremstyle{remark}

\newtheorem{rem}[equation]{Remark}
\newcommand{\refrem}[1]{\textit{Remark \ref{#1}}}

\newtheorem{ass}[equation]{Assumption}
\newcommand{\refass}[1]{\textit{Assumption \ref{#1}}}

\newtheoremstyle{named}{}{}{\itshape}{}{\bfseries}{}{.5em}{#1 #3}
\theoremstyle{named}
\title{Nontangential Maximal Function estimates for the elliptic Mixed Boundary Value Problem with variable coefficients}
\author{Hongjie Dong and Martin Ulmer}
\date{\today}

\begin{document}

\maketitle

\begin{abstract}
We consider an elliptic operator \(L\) with variable, merely bounded, and measurable coefficients on a Lipschitz domain, and study solutions to \(Lu=0\) that attain given Neumann and Dirichlet-regularity data on different parts of the boundary. The boundary data lies in \(L^p\) or \(W^{1,p}\) respectively, and we show nontangential maximal function estimates of the gradient of the solution. This mixed boundary value problem generalizes the pure Dirichlet, regularity, and Neumann problem with rough boundary data in \(L^p\), and the already established mixed boundary value problem for the Laplacian.
\end{abstract}


\section{Introduction}

In this work, we consider a bounded Lipschitz domain \(\Omega\subset\mathbb{R}^n\), where the boundary \(\partial\Omega\) can be written as \(\overline{\mathcal{D}}\cup\mathcal{N}\) for two open components \(\mathcal{D}\) and \(\mathcal{N}\). These two components are intersection free, i.e., \(\mathcal{D}\cap\mathcal{N}=\emptyset\) and we call their shared boundary, \(\Lambda:=\overline{\mathcal{D}}\cap\overline{\mathcal{N}}\), the interface.
Furthermore, we consider the elliptic operator
\begin{align}
    L:=\mathrm{div}(A(x)\nabla \cdot)\label{eq:DefEllipticOperator}
\end{align}
where \(A\) is a non necessarily symmetric matrix with merely bounded and measurable coefficients that satisfies the ellipticity condition, i.e., there exists \(\lambda>0\) such that
\[\lambda|\xi|^2\leq A(x)\xi\cdot \xi,\quad |A(x)|\leq \lambda^{-1}\quad\textrm{for almost every }x\in \Omega, \textrm{ and all }\xi\in\mathbb{R}^n.\]

We are interested in solutions to the mixed boundary value problem, i.e.,  weak solutions \(u\in W^{1,2}_{\text{\text{loc}}}(\Omega)\) to the PDE
\begin{align}\begin{cases}
    Lu=0& \textrm{ in }\Omega,
    \\
    A\nabla u\cdot \nu=g_\mathcal{N} &\textrm{ on }\mathcal{N},
    \\
    u=g_\mathcal{D} &\textrm{ on }\mathcal{D},
\end{cases}\label{eq:MixedPDE}
\end{align}
where \(g_\mathcal{N}\in L^p(\partial\Omega)\) and \(g_\mathcal{D}\in L_1^p(\partial\Omega)\) (see Section \ref{section:def} for the definition of the Haj\l{}asz Sobolev space). The boundary data \(g_\mathcal{D}\) is attained as a nontangential limit in almost every point of \(\mathcal{D}\), while the tangential derivative of \(g_\mathcal{D}\) and the conormal derivative \(g_\mathcal{N}\) are attained in the sense of distributions on \(\mathcal{D}\) or \(\mathcal{N}\) respectively (more details in Section \ref{section:def}).

For instance, PDEs of this type with two different types of boundary condition are useful to describe stationary heat states of objects that are in contact with two different materials on its surface, where one appears as insulating and the other one has perfect thermal conduction. For applications in combustion theory, we would like to refer the reader to \cite{lederman_uniqueness_2001} and \cite{lederman_mixed_2001}, and for applications in modeling exocytosis, we would like to refer to \cite{johansson_solvability_2012}.
\medskip

Our goal is to establish bounds on the nontangential maximal function of the gradient of \(u\) in terms of the boundary data in \(L^p\) (like in \eqref{eq:NontangentialInLqWTSBound}), where the (mean-valued) nontangential maximal function we consider is given by
    \[\tilde{N}(F)(x):=\sup_{y\in\Gamma_\alpha(x)}\Big(\fint_{B_{\delta(y)/2}(y)}|F(z)|^2\,dz\Big)^{1/2}, \]
and \(\Gamma_\alpha(x)\) is a cone of aperture \(\alpha\) with tip in \(x\in\partial\Omega\). See the beginning of Section 3 for the definition. If a weak solution \(u\in W^{1,2}_{\text{loc}}(\Omega)\) satisfies \(\tilde{N}(\nabla u)\in L^p(\partial\Omega)\), then its conormal derivative and its tangential derivative convergence nontangentially almost everywhere and we can make sense of prescribing boundary data like in \eqref{eq:MixedPDE}. Hence, this question is not only interesting from the application perspective, but also from a purely mathematical perspective. 
Originally, Carlos Kenig posed the question of attaining \(L^p\) solvability for mixed boundary value problem for the Laplacian \(L=\Delta\) in \cite[3.2.15]{kenig_harmonic_1994}. Previously, \(L^2\)-solvability was established for the pure Dirichlet (cf. \cite{dahlberg_estimates_1977}) and the pure Neumann problem (cf. \cite{jerison_neumann_1981}), but interestingly, \(L^2\) solvability cannot be obtained for the mixed problem: If we consider the solution \(\mathrm{Re}(\sqrt{z})\) to the Laplace equation on the upper half plane, we see that it attains zero Neumann data on the positive real axis and zero Dirichlet data on the negative real line, but its derivative on the boundary is not locally in \(L^2\).
However, positive results of \(L^2\) solvability could still be obtained for the Laplacian, if one restricts to specific types of domains (cf. \cite{brown_mixed_1994},\cite{BrownSykes2001},\cite{MitreaMitrea2007},\cite{VenouziouVerchota2008}). For more general Lipschitz domains, \(L^p\) solvability was first obtained in \cite{lanzani_mixed_2008} when \(n=2\) for Lipschitz graph domains with small Lipschitz constant less than 1. After that \cite{ott_mixed_2013} proved \(L^p\) solvability for Lipschitz domains with Lipschitz interface in all dimensions, before \cite{taylor_mixed_2013} relaxed the regularity of the interface by only requiring that the interface is \((n-2+\EPS)\)-Alhfors regular for some sufficiently small \(\EPS\) and that \(\mathcal{D}\) satisfies the interior corkscrew condition. Up to this point, the \(L^p\) solvability results were valid for \(p\in(1,p_0)\) for some \(p_0>1\). This \(p_0\) was made precise in \cite{brown_estimates_2021} as \(p_0=\frac{n}{n-1}\), where the authors assumed a more regular \(C^{1,1}\)-domain with Lipschitz interface \(\Lambda\). Due to above example, this solvability range is know to be sharp on Lipschitz domains in \(n=2\). Finally, \cite{dong_dirichlet-conormal_2020} obtains the same \(p_0=\frac{m+2}{m+1}\) for bounded Lipschitz domains with rough interfaces that are essentially perturbations of Lipschitz graphs in \(m\) variables. It is worth noting that their \(p_0\) is independent of the ambient dimension \(n\), and obtains the sharp solvability range \(p_0=2\) if \(m=0\), which means Reifenberg flat interfaces. We essentially work in the setting of their geometric assumption on the interface, and would like to postpone the discussion of the geometrical details on the interface into Section \ref{section:Geometry}. The mixed boundary value problem has also been considered for the Lam\'{e} or the Stokes system (\cite{BrownMitrea2009}, \cite{BrownMitreaMitreaWright2010}).
\medskip

The novelty in this article is that we consider elliptic operators with variable coefficient matrices. Naturally, the theory for these elliptic operators lacks behind the theory for the Laplacian, and is accompanied by new obstacles. Let us denote the pure regularity boundary value problem with data in \(\dot{L}_1^p\) as \((R)_p\) and the pure Neumann boundary value problem with data in \(L^p\) as \((N)_p\) (see Section \ref{section:def} for details). We assume the following:

\begin{ass}[Large DPR condition]\label{ass:largeDPR}
There exists \(C>0\) such that
\begin{align}
    \sup_{\Delta\subset \partial\Omega}\sigma(\Delta)^{-1}\int_{T(\Delta)}\frac{(\mathrm{osc}_{B_{\delta(x)/2}(x)}A)^2}{\delta(x)} dx\leq C,\label{eq:LargeDPRcondition}
\end{align}
    where the supremum is taken over all boundary balls \(\Delta\subset\partial\Omega\), $\delta(x)=\text{dist}(x,\partial\Omega)$, and
    \[\mathrm{osc}_{B_{\delta(x)/2}(x)}A:=\sup_{y,z\in B_{\delta(x)/2}(x)}|A(z)-A(y)|.\]
\end{ass}
Alternatively, we can replace \eqref{eq:LargeDPRcondition} with the more famous \textit{large DKP condition}.
\begin{ass}[Large DKR condition]\label{ass:largeDKP}
There exists \(C>0\) such that
\begin{align}
    \sup_{\Delta\subset \partial\Omega}\sigma(\Delta)^{-1}\int_{T(\Delta)}\sup_{y\in B_{\delta(x)/2}(x)}|\nabla A(y)|^2 \delta(x) dx\leq C.\label{eq:largeDKPcondition}
\end{align}
\end{ass}
Since 
\[\frac{(\mathrm{osc}_{B_{\delta(x)/2}(x)}A)^2}{\delta(x)}\leq\sup_{B_{\delta(x)/2}(x)}|\nabla A(y)|^2 \delta(x),
\] 
we note that \eqref{eq:LargeDPRcondition} is more general than \eqref{eq:largeDKPcondition}.

The large DKP condition has a long and successful history in the study of boundary value problems for variable coefficient elliptic PDEs, and we discuss it after stating our main theorem. It is noteworthy that we only need the DKP condition (\refass{ass:largeDKP}) for the uniqueness part of our main theorem, but not for the first two statements which deal with the existence of solutions to the mixed \(L^p\) boundary value problem. Note also that we postpone the definition of the involved function spaces \(\mathcal{H}^1,\dot{HS}^1, \dot{L}^p_1\) to Section \ref{section:def} and Section \ref{section:ExistenceProof}. Now the main theorem is the following:

\begin{thm}\label{MainThm}
    Let \(\Omega\) be a bounded Lipschitz domain with Lipschitz constant \(l\), \(L\) be an elliptic operator as in \eqref{eq:DefEllipticOperator}, and \(p_0>2\) the reverse Hölder exponent from \reflemma{lemma:RevHolderBoundary}. Assume that the geometric condition \refass{ass:corkscrew} on \(\mathcal{D}\) holds and let \(1<p<\infty\).
    If we assume that \((R)_p\) and \((N)_p\) are solvable, then the following statements are true:

    \begin{enumerate}[(a)]
        \item\label{item:MainThmA} (\(L^1\) existence) There exists \(\EPS_0=\EPS_0(p_0)>0\) such that if \refass{ass:EPS0} holds for \(\EPS_0\), then for all \(g_{\mathcal{N}}\in \mathcal{H}^1(\mathcal{N})\) and \(g_\mathcal{D}\in \dot{HS}_1(\mathcal{D})\), there exist a weak solution \(u\in W_{\text{loc}}^{1,2}(\Omega)\) to the mixed boundary value problem \eqref{eq:MixedPDE}, such that
        \begin{align}\Vert \tilde{N}(\nabla u)\Vert_{L^1(\partial\Omega)}\lesssim \Vert g_{\mathcal{N}}\Vert_{\mathcal{H}^1(\mathcal{N})}+\Vert g_{\mathcal{D}}\Vert_{\dot{HS}^1(\mathcal{D})}.\label{eq:NontangentialInL1WTSBound}\end{align}
        \item\label{item:MainThmC} (\(L^q\) existence) For every \(1<q< \min\{p,p_0/2\}\), there exists \(\EPS_0=\EPS_0(q,p_0)>0\) such that if \refass{ass:EPS0} holds for \(\EPS_0\), then for every \(g_\mathcal{N}\in L^q(\mathcal{N})\) and \(g_\mathcal{D}\in \dot{L}_1^q(\mathcal{D})\), there exist a weak solution \(u\in W_{\text{loc}}^{1,2}(\Omega)\) to the mixed boundary value problem \eqref{eq:MixedPDE}, such that
        \begin{align}\Vert \tilde{N}(\nabla u)\Vert_{L^q(\partial\Omega)}\lesssim \Vert g_\mathcal{N}\Vert_{L^q(\mathcal{N})}+\Vert g_\mathcal{D}\Vert_{\dot{L}_1^q(\mathcal{D})}.\label{eq:NontangentialInLqWTSBound}\end{align}
        \item\label{item:MainThmB} (\(L^1\) uniqueness) If \refass{ass:largeDKP} holds, there exists \(\EPS_0=\EPS_0(l, n,\lambda)>0\) 
        such that if \refass{ass:EPS0} holds for \(\EPS_0\), and if \(u\in W^{1,2}_{\text{loc}}(\Omega)\) is a weak solution to \eqref{eq:MixedPDE} with \(\tilde{N}(\nabla u)\in L^1(\partial\Omega)\), \(g_\mathcal{N}=0\), and \(g_\mathcal{D}=0\), then \(u=0\).

    \end{enumerate}
\end{thm}

The key assumption in \refthm{MainThm} is that we assume, a priori, solvability of \((N)_p\) and \((R)_p\), the pure Neumann and Dirichlet regularity boundary value problem, for some \(1<p<\infty\). Naturally, we can only expect \(L^q\) solvability for the mixed boundary value problem, when the pure problems are already solvable. Interestingly, \eqref{item:MainThmA} and \eqref{item:MainThmC} of \refthm{MainThm} state that these are the only tools needed apart from the geometric assumptions on the interface (see Section \ref{section:Geometry} for details). We obtain \(L^q\) solvability in the range \(q\in [1,\min\{p,p_0/2\})\), which recovers the sharp solvability range result for the Laplacian from \cite{dong_dirichlet-conormal_2020} (see \refrem{rem:DependenciesOfqp0} for details).
\medskip

To apply \refthm{MainThm}, we need to address the question when the pure boundary value problems \((R)_p\) and \((N)_p\) are solvable. This question, in particular for the Dirichlet problem \((D)_p\) (see Section \ref{section:def} for details), has a rich history and one of its most famous sufficient conditions is the (large) DKP-condition. Its study began with \cite{dahlberg_estimates_1977} and \cite{dahlberg_poisson_1979}, where Dahlberg solved \((D)_2\) for the Laplacian in Lipschitz domains and observed that if one flattens the Lipschitz boundary in a conformal way, solving \((D)_2\) for the Laplacian on a Lipschitz domain is equivalent to solving \((D)_2\) for the elliptic operator \(L=\mathrm{div}(A\nabla \cdot)\) on a flat domain. The coefficients of the symmetric matrix \(A\) were only depending on the conformal mapping, and in particular, they satisfied the large DKP condition. In 1984 Dahlberg posed the question whether the solvability of \((D)_2\) or \((D)_p\) for some \(p\) holds for all elliptic operators whose matrices satisfy the DKP condition.
The answer was given in \cite{kenig_dirichlet_2001}, where the authors showed that for any nonsymmetric matrix satisfying the large DKP condition, the elliptic measure \(\omega_L\) lies in the Muckenhoupt space  \(A_\infty(\sigma)\). This is equivalent to saying that \((D)_p\) is solvable for some potentially large \(1<p<\infty\). A little bit later, this result was complemented by \cite{dindos_lp_2007} which introduces the small DKP condition, and shows that for any given \(1<p<\infty\), we have the solvability of \((D)_p\) as long as the DKP constant and the Lipschitz constant of the domain are sufficiently small. More precisely, we can assume the following:

\begin{ass}[Small DPR condition]\label{ass:smallDPR}
    There exists a small \(\EPS>0\) such that
\begin{align}
    \sup_{\Delta\subset \partial\Omega}\sigma(\Delta)^{-1}\int_{T(\Delta)}\frac{(\mathrm{osc}_{B_{\delta(x)/2}(x)}A)^2}{\delta(x)} dx\leq \EPS,\label{eq:SmallDPRcondition}
\end{align}
    where the supremum is taken over all boundary balls \(\Delta\subset\partial\Omega\), and
    \[\mathrm{osc}_{B_{\delta(x)/2}(x)}A:=\sup_{y,z\in B_{\delta(x)/2}(x)}|A(z)-A(y)|.\]
\end{ass}
\begin{ass}[Small DKP condition]\label{ass:smallDKP}
    There exists a small \(\EPS>0\) such that
\begin{align}
    \sup_{\Delta\subset \partial\Omega}\sigma(\Delta)^{-1}\int_{T(\Delta)}\mathrm{sup}_{B_{\delta(x)/2}(x)}|\nabla A|^2\delta(x) dx\leq \EPS,\label{eq:SmallDKPcondition}
\end{align}
    where the supremum is taken over all boundary balls \(\Delta\subset\partial\Omega\), and
    \[\mathrm{osc}_{B_{\delta(x)/2}(x)}A:=\sup_{y,z\in B_\delta(x)/2}|A(z)-A(y)|.\]
\end{ass}
As for the large DPR condition, \eqref{eq:SmallDPRcondition} can be compared to the small DKP condition. 
We can again note that the small DPR condition applies to more operators than the small DKP condition.
\medskip

As turned out, the DKP/DPR conditions also have applications to the regularity and the Neumann problem. Under the DPR condition, the solvability range of the regularity problem is dual to the solvability range of the Dirichlet problem (cf. \cite{dindos_etal_regularity_2023}, \cite{mourgoglou_solvability_2025}). This in particular means that we have small and large DPR condition results for \((R)_p\). The small DPR condition case has previously been dealt with in \cite{dindos_elliptic_2010} in dimension \(n=2\) and later in \cite{dindos_boundary_2017} for all \(n\), where the authors also deal with the small DPR/DKP condition for the Neumann problem \((N)_p\). As for the regularity problem, under the small DPR/DKP conditon, the solvability range for \(p\) for \((N)_p\) is dual to the solvability range of the Dirichlet problem. However, whether the large DKP or DPR condition implies solvability of the Neumann problem \((N)_p\) for some \(p>1\) is still an open question in general. Only in the special case of an unbounded Lipschitz graph domain and in \(n=2\), we know that \((N)_p\) is solvable assuming the large DPR/DKP condition (cf. \cite{dindos_etal_regularity_2023}). However, this argument relies on a change of variable argument that relates the Neumann problem to the regularity problem, but only works in \(n=2\).
\medskip

Now, under the small DPR/DKP condition, we obtain the following corollary from \refthm{MainThm}, which gives a class of variable coefficient matrices for which the \(L^q\) mixed boundary value problem is uniquely solvable:

\begin{cor}\label{MainCorForSmallDKP}
    Let \(\Omega\) be a bounded Lipschitz domain with Lipschitz constant \(l\), \(L\) be an elliptic operator as in \eqref{eq:DefEllipticOperator}, and \(p_0>2\) the reverse Hölder exponent from \reflemma{lemma:RevHolderBoundary}. Let us assume that the geometric condition \refass{ass:corkscrew} on \(\mathcal{D}\) holds and let \(1<p<\infty\). Then there exists \(\EPS=\EPS(p)>0\), such that if \refass{ass:smallDPR} holds for \(\EPS\) and \(l<\EPS\), then the following statements are true:

    \begin{enumerate}[(a)]
        \item \label{item:MainCorA} (\(L^1\) existence) There exists \(\EPS_0=\EPS_0(p_0)>0\) 
        such that if \refass{ass:EPS0} holds for \(\EPS_0\), then for all \(g_{\mathcal{N}}\in \mathcal{H}^1(\mathcal{N})\) and \(g_\mathcal{D}\in \dot{HS}^{1}(\mathcal{D})\), there exist a weak solution \(u\in W_{\text{loc}}^{1,2}(\Omega)\) to the mixed boundary value problem \eqref{eq:MixedPDE} such that
        \[\Vert \tilde{N}(\nabla u)\Vert_{L^1(\partial\Omega)}\lesssim \Vert g_{\mathcal{N}}\Vert_{\mathcal{H}^1(\partial\Omega)}+\Vert g_{\mathcal{D}}\Vert_{\dot{HS}^1(\partial\Omega)}.\]
        \item \label{item:MainCorC} (\(L^q\) existence) For every \(1<q<\min\{p,p_0/2\}\), there exists \(\EPS_0=\EPS_0(q,p_0)>0\) such that if \refass{ass:EPS0} holds for \(\EPS_0\), then for every \(g_\mathcal{N}\in L^q(\partial\Omega)\) and \(g_\mathcal{D}\in \dot{L}_1^q(\partial\Omega)\) there exist a weak solution \(u\in W_{\text{loc}}^{1,2}(\Omega)\) to the mixed boundary value problem \eqref{eq:MixedPDE} such that
        \[\Vert \tilde{N}(\nabla u)\Vert_{L^q(\partial\Omega)}\lesssim \Vert g_\mathcal{N}\Vert_{L^q(\partial\Omega)}+\Vert g_\mathcal{D}\Vert_{\dot{L}_1^q(\partial\Omega)}.\]
        \item \label{item:MainCorB} (\(L^1\) uniqueness) If \refass{ass:smallDKP} holds instead of \refass{ass:smallDPR}, then there exists \(\EPS_0=\EPS_0(l, n,\lambda)>0\) such that if \refass{ass:EPS0} holds for \(\EPS_0\), and if \(u\in W^{1,2}_{\text{loc}}(\Omega)\) is a weak solution to \eqref{eq:MixedPDE} with \(\tilde{N}(\nabla u)\in L^1(\partial\Omega), g_\mathcal{N}=0\) and \(g_\mathcal{D}=0\), then \(u=0\).

    \end{enumerate}
\end{cor}

The article is organized as follows: In Section \ref{section:Geometry}, we discuss the geometric assumption that we need to make on the domain and on the interface, before we make further definitions and introduce the pure Dirichlet, regularity, and Neumann problems in Section \ref{section:def}. Section \ref{section:Prelims} deals with preliminary properties of solutions. After that, Section \ref{section:LocalEstimates} establishes local estimates of the nontangential maximal function that are used in Section \ref{section:ExistenceProof} for the proof of \refthm{MainThm} \eqref{item:MainThmA} and \eqref{item:MainThmC}, the existence of solutions to the \(L^1\) and \(L^q\) mixed boundary value problem. Finally, Section \ref{section:UniquenessProof} contains the proof of the uniqueness of such solutions, i.e., the proof of \refthm{MainThm} \eqref{item:MainThmB}.

\section{Geometry of the domain and interface}\label{section:Geometry}

For an interior point \(x\in \Omega\), we set \(\delta(x):=\mathrm{dist}(x, \partial\Omega)\) as the distance of \(x\) to the boundary, and we set \(\tilde{\delta}(x):=\mathrm{dist}(x, \Lambda)\) as the distance to the interface. We will denote balls centered at \(x\in\mathbb{R}^n\) with radius \(r>0\) as \(B_r(x)\), boundary balls centered at \(x\in \partial\Omega\) by \(\Delta_r(x):=\partial\Omega\cap B_r(x)\), and its Carleson regions as \(T(\Delta_r(x)):=\Omega\cap B_r(x)\).
\smallskip

Next, we define a Lipschitz domain.
\begin{defin}
    The set \(\mathcal{Z}\subset\mathbb{R}^n\) is an \(l\)-cylinder of diameter \(d\) if there exists an orthogonal coordinate system \((y,t)\in\mathbb{R}^{n-1}\times\mathbb{R}\) such that
    \[\mathcal{Z}=\{(y,t):|y|<d,|t|<2ld\}.\]
\end{defin}

\begin{defin}[Lipschitz domain]\label{def:Lipschitzdomain}
    We say \(\Omega\subset\mathbb{R}^n\) is a \textit{Lipschitz domain} with character \((l,N,d)\) if there are \(l\)-cylinders \(\{\mathcal{Z}_j\}_{j=1}^N\) of diameter \(d\) satisfying the following conditions:
    \begin{enumerate}[(i)]
        \item \(8\mathcal{Z}_j\cap\partial\Omega\) is the graph of a Lipschitz function \(\phi_j\), \(\Vert\nabla \phi_j\Vert_\infty\leq l;\phi_j(0)=0\),
        \item \(\partial\Omega=\bigcup_{j=1}^N(\mathcal{Z}_j\cap\partial\Omega)\),
        \item \(\mathcal{Z}_j\cap\Omega\supset\Big\{(x,t)\in \Omega:|x|<d, \delta(x,t)\leq \frac{d}{2}\Big\}.\)
        \item Each cylinder \(\mathcal{Z}_j\) contains points from \(\Omega^C=\mathbb{R}^n\setminus\Omega\).
    \end{enumerate}
\end{defin}

\begin{ass}[Interior corkscrew condition on \(\mathcal{D}\)]\label{ass:corkscrew}
    We assume that \(\mathcal{D}\) satisfies the interior corkscrew condition, i.e., there exists \(R_0,M>0\) such that for all \(x\in \mathcal{D}\) and all \(0<r<R_0\) there exists \(y\in \mathcal{D}\) such that \(\Delta_{r/M}(y)\subset \mathcal{D}\cap \Delta_r(x)\).
\end{ass}

The next assumption depends on \(\EPS_0>0\) which will be chosen sufficiently small later.

\begin{ass}[\(\EPS_0\)]\label{ass:EPS0}
    Assume that there exists \(R_0>0\) and \(\EPS_0\geq 0\) such that for any \(s>-1+\EPS_0,r<R_0\) and \(x\in \Lambda\)
    \begin{align*}
        \fint_{\Delta_r(x)}\tilde{\delta}^s d\sigma\approx r^{s}\quad \textrm{and}\quad \fint_{T(\Delta_r(x))}\tilde{\delta}^{s-1} d\sigma\approx r^{s-1}  .
    \end{align*}
\end{ass}

\begin{rem}
    \refass{ass:EPS0} is an assumption on the geometry of the interface \(\Lambda\). For instance, if \(\Lambda\) is Reifenberg flat, then \refass{ass:EPS0} is satisfied and the flatter the interface is, the smaller is \(\EPS_0\).
\end{rem}

\section{Further definitions}\label{section:def}

For a set \(E\subset\mathbb{R}^n\) we are going to use \(1_E\) as the indicator function on the set \(E\).

 We define a nontangential cone with aperture \(\alpha>1\)  by \[\Gamma_\alpha(x):=\{y\in\Omega: |x-y|<\alpha \delta(y)\}.\]
    The set \(\Gamma_\alpha^\tau(x):=\Gamma_\alpha(x)\cap\{y\in\Omega: \delta(y)\leq \tau\}\) denotes a cone that is truncated at height \(\tau>0\). Furthermore, we set the mean-valued nontangential maximal function as
    \[\tilde{N}_\alpha(F)(x):=\sup_{y\in\Gamma_\alpha(x)}\Big(\fint_{B_{\delta(y)/2}(y)}|F(z)|^2\,dz\Big)^{1/2}. \]
    In addition to that, we define a truncated version of the nontangential maximal function as 
    \[\tilde{N}_{\alpha,\tau}(F)(x):=\sup_{y\in\Gamma^\tau_\alpha(x)}\Big(\fint_{B_{\delta(y)/2}(y)}|F(z)|^2\,dz\Big)^{1/2}, \]
    and the away truncated version for \(\tau>0\) as
    \begin{align}\tilde{N}_{\alpha}^\tau(F)(x):=\sup_{y\in\Gamma_\alpha(x)\cap\{y\in \Omega: \delta(y)\geq \tau \}}\Big(\fint_{B_{\delta(y)/2}(y)}|F(z)|^2\,dz\Big)^{1/2}. \label{def:NontangentialMaximalFucntionTruncatedAway}\end{align}
    When it is clear from context or irrelevant for the argument, we might drop the subscript \(\alpha\) for the aperture.
    
    The Carleson function is defined as
    \begin{align}\tilde{\mathcal{C}}(F)(x):=\sup_{r>0}\frac{1}{|\Delta_r(x)|}\int_{T(\Delta_r(x))}\Big(\fint_{B_{\delta(y)/2}(y)}|F(y,t)|^2\Big)^{1/2}\frac{dy}{\delta(y)}\label{recall:carelsonfunctiondef}\end{align}
    for \(x\in \partial\Omega=\mathbb{R}^n\), and \(|\Delta_r(x)|\approx r^{n-1}\) denotes the surface measure of a boundary ball.

    An important connection between the mean valued nontangential maximal function and the Carleson function is given by the following proposition.
    \begin{prop}[\cite{mourgoglou_solvability_2025}]\label{prop:DualityNontangnetialCarleson}
        For any \(\alpha>1, p\in [1,\infty)\) and for any \(u\) and \(F\) such that \(\Vert \tilde{N}_\alpha(u)\Vert_{L^p(\partial\Omega)}\) and  \(\Vert \tilde{\mathcal{C}}(\delta F)\Vert_{L^{p'}(\partial\Omega)}\) are finite, we have
        \begin{align}\Big|\int_\Omega u F dx\Big|\lesssim_\alpha \Vert \tilde{\mathcal{C}}(\delta F)\Vert_{L^{p'}(\partial\Omega)}  \Vert \tilde{N}_\alpha(u)\Vert_{L^p(\partial\Omega)}. \end{align}
        Moreover, we have
        \begin{align}
            \Vert \tilde{N}_\alpha(u)\Vert_{L^p(\partial\Omega)}\lesssim_\alpha\sup_{F:\Vert \tilde{\mathcal{C}}(\delta F)\Vert_{L^{p'}(\partial\Omega)}=1} \Big|\int_\Omega uF dx\Big|.\label{eq:DualityNontangentialCarelsonFct}
        \end{align} 
    \end{prop}

\subsection{The pure Dirichlet, regularity, and Neumann boundary value problem}

We define a weak solution \(u\in W^{1,2}_{\text{loc}}(\Omega)\) to \(Lu=0\) as a function \(u\) that satisfies
\begin{align}
    \int_\Omega A\nabla u\cdot \nabla \phi dx=0 \quad \textrm{for all }\phi\in W^{1,2}_0(\Omega).\label{eq:StandardPDEwithout}
\end{align}
If a weak solution \(u\) attains the trace \(f\) on the boundary, we say that \(u\) has the Dirichlet boundary data \(f\).

We say a weak solution attains Neumann data \(g\), if
\begin{align}
    \int_\Omega A\nabla u\cdot \nabla \phi dx=\int_{\partial\Omega}g\phi dx=\langle g,\phi\rangle_{\partial\Omega} \quad \textrm{for all }\phi\in W^{1,2}(\Omega),\label{eq:NeummanSolution}
\end{align}
where \(\langle\cdot,\cdot\rangle\) is the dual pairing.

As is classical on Lipschitz domains, Lax-Milgram implies that for every \(g\in W^{-1/2,2}(\partial\Omega)\) there exists a weak solution \(u\in W^{1,2}(\Omega)\) with Neumann data \(g\). Furthermore, for every \(f\in W^{1/2,2}(\partial\Omega)\) there there exists a weak solution \(u\in W^{1,2}(\Omega)\) with Dirichlet data \(f\).
\smallskip

Now, let us define the following three boundary value problems with rough, i.e., \(L^p\), boundary data. 
\begin{defin}[\((D)_p^L\)]
    Let \(p>1\), \(f\in L^p(\partial\Omega)\cap W^{1/2,2}(\partial\Omega)\), and let \(u\in W^{1,2}(\Omega)\) be a weak solution (cf. \eqref{eq:StandardPDEwithout}) with Dirichlet boundary data \(f\). If
    \begin{align}
    \Vert \tilde{N}(u)\Vert_{L^p(\partial\Omega)}\lesssim\Vert f\Vert_{L^p},\label{eq:NontangentialEstimatePureDirichlet}
    \end{align}
    where the implied constants are independent of \(u\) and \(f\), then we say that the \(L^p\) \textit{Dirichlet boundary value problem} is solvable. We write \((D)^L_p\) holds.
\end{defin}
For \(L^p\) Dirichlet boundary data that does not lie in the Bessov space \(W^{1/2,2}(\partial\Omega)\) we cannot say that a function \(u\in W^{1,2}_{\text{loc}}(\Omega)\) attains the boundary data in a trace sense, but instead \eqref{eq:NontangentialEstimatePureDirichlet} guarantees that there exists a solution \(u\in W^{1,2}_{\text{loc}}(\Omega)\) that converges nontangentially to \(f\) almost everywhere.
\smallskip

We can also define the \(L^p\) Neumann boundary value problem.
\begin{defin}[\((N)_p^L\)]\label{def:DefintionOfNeumannProblem}
    Let \(p>1\), \(g\in L^p(\partial\Omega)\cap W^{-1/2,2}(\partial\Omega)\) with \(\int_{\partial\Omega}g d\sigma=0\), and let \(u\in W^{1,2}(\Omega)\) be a weak solution (cf. \eqref{eq:StandardPDEwithout}) with Neumann data \(g\), i.e., \(u\) satisfies \eqref{eq:NeummanSolution}. If
    \begin{align}\Vert \tilde{N}(\nabla u)\Vert_{L^p(\partial\Omega)}\lesssim\Vert g\Vert_{L^p(\partial\Omega)},\label{eq:NontangentialEstiamteInPureNeumann}\end{align}
    where the implied constants are independent of \(u\) and \(g\), then we say that the \(L^p\) \textit{Neumann boundary value problem} is solvable for \(L\). We write \((N)^L_p\) holds. If \(p=1\), we replace the \(L^1\)-norm of \(g\) with the Hardy \(\mathcal{H}^1\)-norm (see \refdef{def:HardySpaces}).
\end{defin}

Furthermore, we call a Borel function \(g:\partial\Omega\to\mathbb{R}\) a \textit{Haj\l{}asz upper gradient} of \(f:\partial\Omega\to\mathbb{R}\) if
\[|f(X)-f(Y)|\leq |X-Y|(g(X)+g(Y))\quad\textrm{for a.e. }X,Y\in \partial\Omega.\]
We denote the collection of all Haj\l{}asz upper gradients of \(f\) as \(\mathcal{D}(f)\) and define \(\dot{L}_1^p(\partial\Omega)\) by all \(f\) with
\[\Vert f\Vert_{\dot{L}^p_1(\partial\Omega)}:= \inf_{g\in\mathcal{D}(f)}\Vert g\Vert_{L^p(\partial\Omega)}<\infty.\]
This space is also called \textit{homogeneous Haj\l{}asz Sobolev space}. In the case of a Lipschitz domain, the \(L^p\) norm of the Haj\l{}asz upper gradient of \(f\) is comparable to the norm of the tangential gradient of \(f\) which is defined almost everywhere on the boundary of a Lipschitz domain (see \cite{mourgoglou_regularity_2024} for more details). The tangential gradient can be written as \(\nabla_Tf(x)=\nabla f(x)\cdot \vec{T}(x)\), where \(\vec{T}\) can be any tangent vector on \(\partial\Omega\). 

Now, we can define the regularity problem with boundary data in \(\dot{L}_1^p(\partial\Omega)\).
\begin{defin}[\((R)_p^L\)]\label{def:DefintionOfRegularityProblem}
    Let \(p>1\), \(f\in \dot{L}^p_{1}(\partial\Omega)\cap W^{1/2,2}(\partial\Omega)\), and let \(u\in W^{1,2}(\Omega)\) be a weak solution (cf. \eqref{eq:StandardPDEwithout}) with Dirichlet boundary data \(f\). If
    \begin{align}\Vert \tilde{N}(\nabla u)\Vert_{L^p(\partial\Omega)}\lesssim\Vert f\Vert_{\dot{L}_{1}^p(\partial\Omega)},\label{eq:RPinequality}\end{align}
    where the implied constants are independent of \(u\) and \(f\), then we say that the \(L^p\) \textit{regularity boundary value problem} is solvable for \(L\). We write \((R)^L_p\) holds for \(L\). If \(p=1\), we replace the \(\dot{L}^1_1\)-norm of \(g\) with the Hardy-Sobolev \(\dot{HS}^1\)-norm (see \refdef{def:HardySpaces}).
\end{defin}
The regularity problem is also called \textit{Dirichlet regularity} problem, since we prescribe Dirichlet data with higher regularity and expect an nontangential estimate on the gradient of the solution instead of merely the solution.
\medskip

In the present setting under assumption of the small DPR condition, \refass{ass:smallDPR}, all three of these boundary value problems are solvable:

\begin{prop}[\cite{dindos_lp_2007} for Dirichlet, \cite{dindos_boundary_2017} for Regularity and Neumann]\label{prop:SolvabilityofDRNforDKP}
    Let \(\Omega\subset \mathbb{R}^n\) be a bounded Lipschitz domain with Lipschitz constant \(l\) (cf. \refdef{def:Lipschitzdomain}). For every \(1<p<\infty\) there exists a sufficiently small \(\EPS=\EPS(n,\lambda,p)\), such that if the small DPR condition, \refass{ass:smallDPR}, holds for that \(\EPS\) and \(l\leq \EPS\), then \((D)_{p'},(R)_p,\) and \((N)_p\) are solvable.
\end{prop}

\subsection{Convergence to the boundary data in the distributional sense}

It is clear that if \(g_\mathcal{D}\in W^{1/2,2}(\partial\Omega)\), then the solution \(u\in W^{1,2}(\Omega)\) to the Dirichlet problem has \(g_\mathcal{D}\) as trace. Similarly, if \(g_\mathcal{N}\in W^{-1/2,2}(\partial\Omega)\), then for the solution \(u\in W^{1,2}(\Omega)\) to the Neumann problem \(A\nabla u\cdot \nu=g_\mathcal{N}\) can be understood in the weak sense. However, since we would like to consider solutions with rougher \(L^p\)-boundary data, we need to define convergence of \(A\nabla u\cdot\nu\) to \(g_\mathcal{N}\) or \(\nabla_Tu\) to \(\nabla_Tg_\mathcal{D}\) in the sense of distributions: For each cylinder \(\mathcal{Z}\) with coordinates such that \(\mathcal{Z}\cap\Omega=\{(x,t)\in \mathbb{R}^{n-1}\times\mathbb{R}:t>\phi(x),|x|<d,\delta(x,t)\leq \frac{d}{2}\}\) for a Lipschitz function \(\phi\), we can without loss of generality assume that there exists a smooth function \(\theta_\mathcal{Z}\in C^\infty_0(\mathcal{Z})\) such that \(\sum_{\mathcal{Z}}\theta_\mathcal{Z}|_{\partial\Omega}\equiv 1\). For \(u\in W^{1,2}_{loc}(\Omega)\), we say that \(A\nabla u\cdot\nu\) converges to \(g_\mathcal{N}\in L^p(\partial\Omega)\) in the sense of distributions, if
\begin{align}
\lim_{h\to 0}\int_{\partial\Omega}\Big(\fint_{B_{h/2}(x,t+h)}(A\nabla u)dy\Big)\cdot \nu(x,t) \Phi(x,t)\theta_\mathcal{Z}(x,t) d\sigma(x,t)\label{eq:weakConvNeumannData}
\\=\int_{\mathcal{Z}\cap\partial\Omega}g_\mathcal{N}(x,t)\Phi(x,t)\theta_\mathcal{Z}(x,t) d\sigma(x,t)\nonumber\end{align}
for all \(\Phi\in \mathcal{S}(\mathbb{R}^n)\), where \(\mathcal{S}\) is the class of Schwarz functions. Since our domain is bounded, this coincides with taking all \(\Phi\in C^\infty(\overline{\Omega})\). Furthermore, we say \(\nabla_T u\) converges to \(\nabla_Tg_\mathcal{D}\) in the sense of distributions, if 
\begin{align}
\lim_{h\to 0}\int_{\partial\Omega}\Big(\fint_{B_{h/2}(x,t+h)}\nabla udy\Big)\cdot \vec{T}(x,t) \Phi(x,t)\theta_\mathcal{Z}(x,t) d\sigma(x,t)\label{eq:weakConvRegData}\\
=\int_{\mathcal{Z}\cap\partial\Omega}\nabla _T g_\mathcal{D}\Phi(x,t)\theta_\mathcal{Z}(x,t) d\sigma(x,t)\nonumber\end{align}
for all \(\Phi\in C^\infty(\overline{\Omega})\). For the mixed boundary value problem, distributional convergence of \(A\nabla u\cdot \nu\) to the Neumann data \(g_\mathcal{N}\) on \(\mathcal{N}\) and \(\nabla_Tu\) to \(\nabla_Tg_\mathcal{D}\) on \(\mathcal{D}\) means that 
\eqref{eq:weakConvNeumannData} holds for test functions \(\Phi\) with \(\Phi|_\mathcal{D}=0\) and \eqref{eq:weakConvRegData} holds for test functions \(\Phi\) with \(\Phi|_\mathcal{N}=0\)
\smallskip

If the pure Neumann problem \((N)_p\) is solvable, i.e., \eqref{eq:NontangentialEstiamteInPureNeumann} holds, then for every \(g\in L^p(\partial\Omega)\) (replacing with $\mathcal H^1(\partial\Omega)$ when $p=1$) there exist a solution \(u\in W^{1,2}_{\text{loc}}(\Omega)\) that attains the Neumann data \(g\) in the sense of distributions. This follows from the proof of \cite[Theorem 3.1]{kenig_neumann_1993}, which can be generalized to Lipschitz domains. 
\smallskip

Similarly and also by the proof of \cite[Theorem 3.1]{kenig_neumann_1993}, if the pure regularity problem \((R)_p\) is solvable, i.e., \eqref{eq:RPinequality} holds, then for every \(f\in \dot{L}^p_1(\partial\Omega)\) (replacing with $\dot{HS}^1(\partial\Omega)$ when $p=1$) there exist a solution \(u\in W^{1,2}_{\text{loc}}(\Omega)\) that attains the Dirichlet data \(f\) nontangentially almost everywhere, and \(\nabla_T u\) converges to \(\nabla_Tf\) in the sense of distributions.
\smallskip

For the mixed problem, the same arguments imply that if the bound \eqref{eq:NontangentialInLqWTSBound} holds for all \(g_\mathcal{D}\in W^{1/2,2}(\mathcal{D})\cap\dot{L}_1^p(\mathcal{D})\) and \(g_\mathcal{N}\in W^{-1/2,2}(\mathcal{N})\cap L^p(\mathcal{N})\), then for every \(g_\mathcal{D}\in \dot{L}_1^p(\mathcal{D})\) and \(g_\mathcal{N}\in L^p(\mathcal{N})\), there exists a weak solution \(u\in W^{1,2}_{loc}(\Omega)\) to the mixed boundary value problem such that \(A\nabla u\cdot\nu\) converges to \(g_\mathcal{N}\) in the sense of distributions, that \(u\) converges to \(g_\mathcal{D}\) nontangentially almost every where on \(\mathcal{D}\), and that \(\nabla_Tu\) converges to \(\nabla_Tg_\mathcal{D}\) on \(\mathcal{D}\) in the sense of distributions. If \(p=1\), the same statement holds assuming \eqref{eq:NontangentialInL1WTSBound} and replacing \(\dot{L}_1^1(\mathcal{D})\) and \(L^1(\mathcal{N})\) by the corresponding Hardy and Hardy-Sobolev spaces \(\dot{HS}^1(\mathcal{D})\) and \(\mathcal{H}^1(\mathcal{N})\).

\medskip

\section{Preliminaries: Properties of solutions}\label{section:Prelims}

The following two results are standard for solutions with either pure Dirichlet or pure Neumann boundary data. However, their proofs carry through to the case of mixed boundary data completely analogously, as was observed at the beginning of Section 3 in \cite{taylor_mixed_2013}. Since we only work on Lipschitz domains, we do not aim to state the following results on domains with more general geometry (these results can be found for instance in \cite{gilbarg_elliptic_2001} for pure Dirichlet, or \cite{hofmann_neumann_2025}, \cite{david_dimension_2024} for pure Neumann).

\begin{prop}[Moser estimate]\label{prop:MoserEstimate}
Let \(\Omega\subset\mathbb{R}^n\) be a Lipschitz domain and \(L=\mathrm{div}(A\nabla \cdot)\) an elliptic operator. Then there exists \(C>0\) such that
\begin{enumerate}[(i)]
    \item for any ball \(B=B_r(x)\) such that \(2B\subset \Omega\) and any solution \(u\in W^{1,2}(2B)\) to \(Lu=0\) in \(2B\) we have
    \begin{align}
        \mathrm{osc}_B u \lesssim C\fint_{2B}|u-(u)_B|dx;
    \end{align}\label{item:MoserInterior}
    \item for any boundary ball \(\Delta=\Delta_r(x)\) such that \(x\subset \partial\Omega\) and any solution \(u\in W^{1,2}(T(2\Delta))\) to \(Lu=0\) in \(T(2\Delta)\) with zero Neumann data on \(2\Delta\) we have
    \begin{align}
        \mathrm{osc}_{T(\Delta)} u \lesssim C\fint_{T(2\Delta)}|u-(u)_{T(2\Delta)}|dx;
    \end{align}\label{item:MoserNeumann}
    \item for any boundary ball \(\Delta=\Delta_r(x)\) such that \(x\subset \partial\Omega\) and any solution \(u\in W^{1,2}(T(2\Delta))\) to \(Lu=0\) in \(T(2\Delta)\) with zero Dirichlet data on \(2\Delta\cap \mathcal{D}\) and zero Neumann data on \(2\Delta\cap\mathcal{N}\) and \(\frac{3}{2}\Delta\cap\mathcal{D}\neq \emptyset\) we have
    \begin{align}
        \mathrm{osc}_{T(\Delta)} u \lesssim C\fint_{T(2\Delta)}|u|dx.
    \end{align}\label{item:MoserMixed}
\end{enumerate}
\end{prop}

\begin{prop}[Boundary H\"{o}lder]\label{prop:BoundaryHoelder}
Let \(\Omega\subset\mathbb{R}^n\) be a Lipschitz domain and \(L=\mathrm{div}(A\nabla \cdot)\) an elliptic operator. Then there exists \(C>0\) and \(0<\beta<1\) such that for all \(0<\EPS<1/2\) and for any \(x\subset \partial\Omega\), any boundary ball \(\Delta=\Delta_r(x)\), and any solution \(u\in W^{1,2}(T(\Delta))\) to \(Lu=0\) in \(T(\Delta)\) with zero Dirichlet data on \(\Delta\cap \mathcal{D}\) and zero Neumann data on \(\Delta\cap\mathcal{N}\), we have
    \begin{align}
        \mathrm{osc}_{T(\EPS\Delta)} u \leq C\EPS^\beta \mathrm{osc}_{T(\Delta)} |u|.
    \end{align}

\end{prop}

\begin{lemma}[Reverse H\"{o}lder inequality]\label{lemma:RevHolderBoundary}
    Suppose \refass{ass:corkscrew}. Let \(u\in W^{1,2}(T(\Delta_{2r}))\) be a solution to \(Lu=0\) with zero Dirichlet data on \(\mathcal{D}\cap\Delta_{2r}\) and Neumann data \(g_\mathcal{N}\) on \(\Delta_{2r}\cap\mathcal{N}\), then there exists \(p_0>2\) and \(C>0\) such that
    \begin{align}
    \Big(\fint_{T(\Delta_r)}|\nabla u|^{p_0}dx\Big)^{\frac{1}{p_0}}\lesssim \fint_{T(\Delta_{2r})}|\nabla u|dx + C\Big(\fint_{\Delta_{2r}}1_\mathcal{N}|g_{\mathcal{N}}|^{p_0\frac{n-1}{n}}dx\Big)^{\frac{n}{p_0(n-1)}}.\label{eq:lemma:RevHolderBoundary}
    \end{align}
\end{lemma}
The proof of the boundary reverse Hölder's inequality is an entirely real variable argument which requires a boundary Cacciopolli inequality as the only tool that stems from the PDE. Hence, its proof is completely analogous to the proof of Lemma 3.19 in \cite{ott_mixed_2013}.
\begin{rem}
    The numbers \(p_0\) and \(C\) in \reflemma{lemma:RevHolderBoundary} depend on the constants in the interior corkscrew condition of \(\mathcal{D}\) (\refass{ass:corkscrew}), the ambient dimension \(n\), and the ellipticity constant \(\lambda\). In general, for variable coefficient operators \(L\) like in \eqref{eq:DefEllipticOperator}, \(p_0\) is bigger than but close to 2. However under additional flatness conditions, \cite{choi_optimal_2021} shows the reverse Hölder's inequality for all \(p_0\in (3/4,4)\). These conditions are that the boundary of the domain \(\partial\Omega\) and the interface \(\Lambda\) are Reifenberg flat with constants depending on \(p_0\) and that the coefficients of the operator have sufficiently small BMO semi-norm depending on \(p_0\).
\end{rem}

\subsection{Existence of \texorpdfstring{\(W^{1,2}\)}{} solutions for the mixed boundary value problem}

In the following let \(W^{1/2,2}(\mathcal{D})\) be the restriction of the trace space of \(W^{1,2}(\Omega)\) on \(\mathcal{D}\subset\partial\Omega\), and \(W^{-1/2,2}(\mathcal{N})\) be the restriction of the trace space of the normal derivative of functions in \(W^{1,2}(\Omega)\) on \(\mathcal{N}\subset\partial\Omega\).

\begin{thm}\label{thm:ExistenceofW12sol}
    Let \(g_\mathcal{D}\in W^{1/2,2}(\mathcal{D}), g_\mathcal{N}\in W^{-1/2,2}(\mathcal{N})\). Then there exists a unique weak solution \(u\in W^{1,2}(\Omega)\) to \eqref{eq:MixedPDE} such that
    \[\Vert u\Vert_{W^{1,2}(\Omega)}\lesssim \Vert g_\mathcal{D}\Vert_{W^{1/2,2}(\mathcal{D})} + \Vert g_\mathcal{N}\Vert_{W^{-1/2,2}(\mathcal{N})}.\]
\end{thm}

This proof follows from Lax-Milgram and is fairly standard, hence we omit it.

\section{Local estimates for the nontangential maximal function}\label{section:LocalEstimates}

In this section, we establish important local estimates for the nontangnential maximal function that we will use in the proof of \refthm{MainThm} \eqref{item:MainThmA} and \eqref{item:MainThmC}.

First, since we assume solvability of \((R)_p\) and \((N)_p\) for some \(p>1\), without loss of generality, we can assume from now on, that \(u\in W_{\text{loc}}^{1,2}(\Omega)\) is a weak solution to 
\begin{align}\begin{cases}
    Lu=0& \textrm{ in }\Omega,
    \\
    A\nabla u\cdot \nu=g_\mathcal{N} &\textrm{ on }\mathcal{N},
    \\
    u=0 &\textrm{ on }\mathcal{D}.
\end{cases}\label{PDEwithZeroDirichletdata}
\end{align}

We fix this \(p\) for the rest of this section. We will consider a small boundary ball \(\Delta_r\) and have to distinguish three cases: The pure Neumann, the pure Dirichlet, and the mixed case.

\begin{lemma}\label{lemma:LocalNontangnetialEstimatePureNeumann}
    Let \(\Delta=\Delta_r\) be a boundary ball with \(\Delta_{2r}\subset\mathcal{N}.\) Then we have
    \[\Vert \tilde{N}_{r/2}(\nabla u)\Vert_{L^p(\Delta_r)}\lesssim \Vert g_\mathcal{N}\Vert_{L^p(\Delta_{2r})} + r^{\frac{n-1}{p}-n}\int_{T(\Delta_{2r})}|\nabla u| dx.\]
\end{lemma}

\begin{proof}
    Let \(w\in W_{\text{loc}}^{1,2}(\Omega)\) be the solution to \(Lw=0\) with Neumann boundary data \(g_\mathcal{N}1_{\Delta_{2r}}\) on \(\partial\Omega\). Then
     \[\Vert \tilde{N}_{r/2}(\nabla u)\Vert_{L^p(\Delta_r)}\lesssim \Vert \tilde{N}_{r/2}(\nabla (u-w))\Vert_{L^p(\Delta_r)} + \Vert \tilde{N}_{r/2}(\nabla w)\Vert_{L^p(\Delta_r)}.\]
     By the solvability of \((N)_p\) (cf. \refprop{prop:SolvabilityofDRNforDKP}), we have for the second term \[\Vert \tilde{N}_{r/2}(\nabla w)\Vert_{L^p(\Delta_r)}\lesssim \Vert g_\mathcal{N}\Vert_{L^p(\Delta_{2r})}.\]
     Since \(u-w\) has zero Neumann data in \(\Delta_{2r}\), we can apply Lemma 4.1 in \cite{feneuil_lp_2024}. Note here, that solvability of \((N)_p\) implies solvability of the weak Poisson Neumann problem \((wPN)_{p'}\), which is introduced in \cite{feneuil_lp_2024}. This implication is also proved in \cite{feneuil_lp_2024} and allows us to apply their Lemma 4.1. We obtain
     \begin{align*}\Vert \tilde{N}_{r/2}(\nabla(u- w))\Vert_{L^p(\Delta_r)}&\lesssim r^{\frac{n-1}{p}-n}\int_{T(\Delta_{2r})}|\nabla (u-w)|dx
     \\
     &\lesssim r^{\frac{n-1}{p}-n}\int_{T(\Delta_{2r})}|\nabla u|dx +r^{\frac{n-1}{p}-n}\int_{T(\Delta_{2r})}|\nabla w|dx.\end{align*}
     For the second term, we can use Hölder's inequality and the solvability of \((N)_p\) (cf. \refprop{prop:SolvabilityofDRNforDKP}), and we get
     \begin{align*}r^{\frac{n-1}{p}-n}\int_{T(\Delta_{2r})}|\nabla w|dx&\leq \Big(r^{n-1}\fint_{T(\Delta_{2r})}|\nabla w|^pdx\Big)^{1/p}\lesssim \Vert \tilde{N}_{2r}(\nabla w)\Vert_{L^p(\Delta_{2r})}
     \\
     &\lesssim \Vert g_\mathcal{N}\Vert_{L^p(\Delta_{2r})}.\end{align*}
\end{proof}

\begin{lemma}\label{lemma:LocalNontangnetialEstimatePureDirichlet}
    Let \(\Delta=\Delta_r\) be a boundary ball with \(\Delta_{2r}\subset\mathcal{D}.\) Then we have
    \[\Vert \tilde{N}_{r/2}(\nabla u)\Vert_{L^p(\Delta_r)}\lesssim r^{\frac{n-1}{p}-n}\int_{T(\Delta_{2r})}|\nabla u| dx.\]
\end{lemma}

\begin{proof}
    As in \cite{feneuil_lp_2024}, by the duality \eqref{eq:DualityNontangentialCarelsonFct} and an approximation argument, there exists a function \(h\in C_c(T(\Delta_{\frac{3}{2}r}))\) with \(\Vert \tilde{\mathcal{C}}(\delta h)\Vert_{L^{p'}(\partial\Omega)}=1\) such that
    \begin{align*}
        \Vert \tilde{N}_{r/2}(\nabla u)\Vert_{L^p(\partial\Omega)}\lesssim \int_\Omega \nabla u \cdot h dx.
    \end{align*}
    Let \(\phi\in C^\infty_c(T(\Delta_{\frac{14}{8}r}))\) be a cut-off function, i.e., \(\phi\equiv 1\) on \(T(\Delta_{\frac{13}{8}r})\) and \(|\nabla \phi|\lesssim \frac{1}{r}\). Then by integration by parts, we obtain
    \begin{align*}
        \int_\Omega \nabla u \cdot h dx=\int_\Omega \nabla u \cdot h \phi dx = -\int_\Omega  u h \cdot \nabla\phi dx - \int_\Omega u \phi \, \mathrm{div}(h) dx=:I_1+I_0.
    \end{align*}
    Since \(\nabla\phi\equiv 0\) on \(\Delta_{\frac{3}{2}r}=\mathrm{\supp}(h)\), we obtain \(I_1=0\).
    Now, let \(v\in W^{1,2}(\Omega)\) be the solution to the Poisson-Dirichlet problem
    \[\begin{cases} L^*v=\mathrm{div}(h) & \textrm{in }\Omega, \\ v=0 &\textrm{on }\partial\Omega.\end{cases}\]
    Then we can continue \(I_0\) with
    \begin{align*}
        I_0&=-\int_\Omega u\phi \, L^*v dx\\
        &=\int_\Omega A\nabla u \cdot \nabla v \, \phi dx + \int_\Omega A\nabla \phi \cdot \nabla v \, u dx -\int_{\partial\Omega}u\phi \, A^*\nabla v\cdot \nu  d\sigma
        \\
        &=:I_{00}+I_2+B.
    \end{align*}
    Since \(u\phi\) vanishes on \(\partial\Omega\), the last term \(B\) is zero. By use of the PDE \(Lu=0\) we obtain for the first term
    \begin{align*}
        I_{00}=\int_\Omega A\nabla u\cdot\nabla (v\phi) dx - \int_\Omega A\nabla u\cdot \nabla \phi \, v dx= -\int_\Omega A\nabla u\cdot \nabla \phi \, v dx=:I_3.
    \end{align*}
    We now bound the remaining terms \(I_2\) and \(I_3\) all with similar methods.
    
    First, for \(I_2\), we have by the Moser estimate for \(u\) 
    \begin{align*}
        I_2&\lesssim \frac{1}{r}\Big(\int_{T(\Delta_{\frac{14}{8}r})\setminus T(\Delta_{\frac{13}{8}r})}|\nabla v|dx\Big)\sup_{T(\Delta_{\frac{14}{8}r})}|u|
        \\
        &\lesssim r^{n-1}\Big(\fint_{{T(\Delta_{\frac{14}{8}r})\setminus T(\Delta_{\frac{13}{8}r})}}|\nabla v|dx\Big)\fint_{T(2\Delta)}|u|dx.
    \end{align*}
    Since \(v\) is a solution to the homogeneous Poisson-Dirichlet problem \(L^*v=0\) on \(T(\Delta_{2r})\setminus T(\Delta_{\frac{3}{2}r})\), we can apply the boundary Cacciopolli inequality to bound
    \begin{align*}r^{n-1}\fint_{{T(\Delta_{\frac{14}{8}r})\setminus T(\Delta_{\frac{13}{8}r})}}|\nabla v|dx\lesssim r^{n-1}\Big(\fint_{{T(\Delta_{\frac{14}{8}r})\setminus T(\Delta_{\frac{13}{8}r})}}|\nabla v|^2dx\Big)^{1/2}
    \\
    \lesssim r^{n-2}\Big(\fint_{T(\Delta_{\frac{15}{8}r})\setminus T(\Delta_{\frac{3}{2}r})}|v|^2dx\Big)^{1/2}.\end{align*}
    By Moser estimates (\refprop{prop:MoserEstimate}), we obtain
    \begin{align*}
        r^{n-2}\Big(\fint_{T(\Delta_{\frac{15}{8}r})\setminus T(\Delta_{\frac{3}{2}r})}|v|^2dx\Big)^{1/2}\lesssim r^{n-2}\fint_{T(\Delta_{2r})}|v|dx\lesssim r^{-2}\int_{T(\Delta_{2r})}|v|dx.
    \end{align*}
    We now split the solid integral into integrating over cones and then the boundary (cf. Lemma 3.14 in \cite{milakis_harmonic_2013}), and obtain
    \begin{align}
        r^{-2}\int_{T(\Delta_{2r})}|v|dx&\lesssim r^{-2}\int_{\Delta_{2r}}\Big(\int_{\Gamma^{2r}(x)}|v(y)|\delta(y)^{-n+1}dy\Big)  d\sigma(x)\nonumber
        \\
        &\lesssim r^{-2}\int_{\Delta_{2r}}\tilde{N}(v)(x)\Big(\int_{\Gamma^{2r}(x)}\delta(y)^{-n+1}dy\Big)  d\sigma(x).\nonumber
    \end{align}
    Note that the integral over the truncated cone yields \(\int_{\Gamma^{2r}(x)}\delta(y)^{-n+1}dy\approx r\). We obtain by Hölder's inequality
    \[r^{-1}\int_{\Delta_{2r}}\tilde{N}(v)d\sigma\lesssim r^{-1+(n-1)/p}\big(\int_{\Delta_{2r}}\tilde{N}(v)^{p'}d\sigma\Big)^{1/p'}.\]
    By Theorem 1.22 in \cite{mourgoglou_solvability_2025}, the solvability of \((D)_{p'}\) is equivalent to the solvability of the Poisson-Dirichlet problem \((PD)_{p'}\) (introduced in Definition 1.17 in \cite{mourgoglou_solvability_2025}), which gives the estimate 
    \[\Vert \tilde{N}(v)\Vert_{L^{p'}(\partial\Omega)}\lesssim \Vert \tilde{\mathcal{C}}(\delta |h|)\Vert_{L^{p'}}=1,
    \] hence in total
    \begin{align}
        r^{n-2}\fint_{T(\Delta_{2r})}|v|dx&\lesssim r^{-1+\frac{n-1}{p}},\label{eq:ArgForv}
    \end{align}
    and by the boundary Poincar\@{e} inequality
    \[
    I_2\lesssim r^{\frac{n-1}{p}-n}\int_{T(\Delta_{2r})}|\nabla u| dx\] follows.
    
    Lastly, for \(I_3\), we have by Moser estimates (\refprop{prop:MoserEstimate}) for \(v\) in the set where \(Lv=0\)
    \begin{align*}
        I_3&\lesssim \frac{1}{r}\Big(\int_{T(\Delta_{2r})}|\nabla u|dx\Big)\sup_{T(\Delta_{\frac{14}{8}r}\setminus \Delta_{\frac{13}{8}r} )}|v|
        \\
        &\lesssim r^{-1}\Big(\int_{T(\Delta_{2r})}|\nabla u|dx\Big)\Big(\fint_{T(\Delta_{2r})}|v|dx\Big)
        \lesssim r^{\frac{n-1}{p}-n}\int_{T(\Delta_{2r})}|\nabla u| dx,
    \end{align*}
    where the last inequality follows from \eqref{eq:ArgForv}. This completes the proof.  

\end{proof}

\begin{lemma}\label{lemma:LocalNontangnetialEstimateMixed}
    Let \(\Delta=\Delta_r\) be a boundary cube with \(\mathrm{dist}(\Delta_r,\Lambda)\leq 2r\), and let \(p_0>2\) from \reflemma{lemma:RevHolderBoundary}. For every \(1<q<\min(p, p_0/2)\) there exists \(\EPS_0=\EPS_0(p_0,q)>0\) such that if \refass{ass:EPS0} holds with \(\EPS_0\), then there exists \(C>0\) such that
    \[\Big(\fint_{\Delta_r} \tilde{N}_{r/2}(\nabla u)^qdx\Big)^{1/q}\leq C\Vert g_\mathcal{N}\Vert_{L^\infty(\Delta_{3r})}+C\Big(\fint_{T(\Delta_{3r})}|\nabla u|^2 dx\Big)^{1/2}.\]
\end{lemma}

\begin{proof}
    Recall that \(p>1\) was chosen such that the Neumann problem \((N)_q\) is solvable for all \(1<q\leq p\). Then, we get by \reflemma{lemma:LocalNontangnetialEstimatePureNeumann},
    \begin{align*}
        \Big(\fint_{\Delta_r} \tilde{N}_{r/2}(\nabla u)^qdx\Big)^{1/q}\lesssim \Big(\fint_{\Delta_{2r}} |A\nabla u\cdot \nu|^qdx\Big)^{1/q} + \fint_{T(\Delta_{2r})}|\nabla u| dx.
    \end{align*}
    We will bound the first term by 
    \begin{equation}
        \label{eq:WTSestimate}
    \begin{aligned}
     &   \Big(\fint_{\Delta_{2r}} |A\nabla u\cdot \nu|^qdx\Big)^{1/q}\lesssim \Big(\fint_{T(\Delta_{2r})}|\nabla u|^2 dx\Big)^{1/2} + \Vert g_N\Vert_{L^\infty(\Delta_{2r})},
    \end{aligned}
    \end{equation}
    which yields the statement of the lemma. Hence it remains to prove the inequality in \eqref{eq:WTSestimate}.

    For this, first we let \(s>0\) be a parameter to be chosen later. By H\"{o}lder's inequality, we have
    \begin{align}
        \Big(\fint_{\Delta_r}|A\nabla u\cdot \nu|^qdx\Big)^{1/q}\lesssim \Big(\fint_{\Delta_r}|A\nabla u\cdot \nu|^{p}\tilde{\delta}^{1-s}dx\Big)^{\frac{1}{p}}\Big(\fint_{\Delta_r}\tilde{\delta}^{-\frac{(1-s)q}{p-q}}dx\Big)^{\frac{1}{q}-\frac{1}{p}}.\label{eq:Proof_4}
    \end{align}
    For the first term, we decompose \(\Delta_r\) into boundary Whitney cubes \(W_i\) with respect to the interface \(\Lambda\). These Whitney balls are chosen such that the collection of enlargements \(2W_i\) cover \(\Delta_{2r}\) with finite overlap, i.e., there exists \(M\in \mathbb{N}\) such that \(\chi_{\Delta_r}\leq\sum_{i}\chi_{W_i}\leq \sum_{i}\chi_{2W_i}\leq M \chi_{\Delta_{2r}}\), and that \(W_i\) is a boundary ball with \(l(W_i)\approx \mathrm{dist}(W_i,\Lambda)\).

    Since the enlargements \(2W_i\) of each of the Whitney balls \(W_i\) still lie either entirely in \(\mathcal{D}\) or \(\mathcal{N}\), we can apply \reflemma{lemma:LocalNontangnetialEstimatePureDirichlet} or \reflemma{lemma:LocalNontangnetialEstimatePureNeumann} to obtain
    \begin{align*}
        &\Big(\fint_{\Delta_r}|A\nabla u\cdot \nu|^{p}\tilde{\delta}^{1-s}dx\Big)^{\frac{1}{p}}=\Big(r^{-n+1}\sum_i\int_{W_i}|A\nabla u\cdot \nu|^{p}\tilde{\delta}^{1-s}dx\Big)^{\frac{1}{p}}
        \\
        &\leq \Big(r^{-n+1}\sum_il(W_i)^{1-s}\int_{W_i}\tilde{N}_{l(W_i)/2}(|\nabla u|)^{p}dx\Big)^{\frac{1}{p}}
        \\
        &\lesssim \Big[r^{-n+1}\sum_il(W_i)^{1-s}\Big(\big(\int_{2W_i}\chi_\mathcal{N}g_\mathcal{N}^{p}dx\big)^{1/p} + l(W_i)^{(n-1)/p}\fint_{T(2W_i)}|\nabla u|dx\Big)^{p}\Big]^{\frac{1}{p}}.
    \end{align*}
    We can simplify with Jensen's inequality and further bound this by
    \[
        \lesssim \left[\Vert\chi_\mathcal{N} g_\mathcal{N}\Vert_{L^\infty(\Delta_{2r})}^p\fint_{\Delta_{2r}}\tilde{\delta}^{1-s}dx + r\fint_{T(\Delta_{2r})}\tilde{\delta}^{-s}|\nabla u|^pdx\right]^{\frac{1}{p}}.
        \]
    Without loss of generality, we can assume \(p<p_0\) and we can bound the previous expression by
    \begin{align}
        &\lesssim \Vert \chi_\mathcal{N}g_\mathcal{N}\Vert_{L^\infty(\Delta_{2r})}\Big(\fint_{\Delta_{2r}}\tilde{\delta}^{1-s}dx\Big)^{1/p} + \Big(r\fint_{T(\Delta_{2r})}\tilde{\delta}^{-s}|\nabla u|^pdx\Big)^{\frac{1}{p}}\nonumber
        \\
        &\lesssim \Vert \chi_\mathcal{N}g_\mathcal{N}\Vert_{L^\infty(\Delta_{2r})}\Big(\fint_{\Delta_{2r}}\tilde{\delta}^{1-s}dx\Big)^{1/p} + r^{1/p}\Big(\fint_{T(\Delta_{2r})}|\nabla u|^{p_0}dx\Big)^{\frac{1}{p_0}}\Big(\fint_{T(\Delta_{2r})}\tilde{\delta}^{-s\frac{p_0}{p_0-p}}dx\Big)^{\frac{p_0-p}{p_0p}}\label{eq:Proof_1},
    \end{align}
    where we used Hölder's inequality in the last line and \(p_0\) is the exponent of the reverse Hölder's inequality (\reflemma{lemma:RevHolderBoundary}). Now, \reflemma{lemma:RevHolderBoundary} yields
    \begin{align}\Big(\fint_{T(\Delta_{2r})}|\nabla u|^{p_0}dx\Big)^{\frac{1}{p_0}}&\lesssim \fint_{T(\Delta_{3r})}|\nabla u|dx + \Vert \chi_\mathcal{N}g_\mathcal{N}\Vert_{L^\infty(\Delta_{3r})}\nonumber
    \\
    &\lesssim \Big(\fint_{T(\Delta_{3r})}|\nabla u|^2dx\Big)^{1/2} +\Vert\chi_\mathcal{N} g_\mathcal{N}\Vert_{L^\infty(\Delta_{3r})}.\label{eq:Proof_2}\end{align}
    Hence, we obtain in total, collecting the estimates in \eqref{eq:Proof_4}, \eqref{eq:Proof_1}, and \eqref{eq:Proof_2} that
    \begin{align}
        \Big(\fint_{\Delta_r}|A\nabla u\cdot \nu|^qdx\Big)^{1/q}
        \lesssim \Big(\fint_{\Delta_r}\tilde{\delta}^{-\frac{(1-s)q}{p-q}}dx\Big)^{\frac{1}{q}-\frac{1}{p}}\Big[\Vert \chi_{\mathcal{N}}g_N\Vert_{L^\infty(\Delta_{2r})}\Big(\fint_{\Delta_{2r}}\tilde{\delta}^{1-s}dx\Big)^{1/p}\nonumber
        \\
        + r^{1/p}\Big(\Vert \chi_\mathcal{N}g_\mathcal{N}\Vert_{L^\infty(\Delta_{3r})}
        + \Big(\fint_{T(\Delta_{3r})}|\nabla u|^2dx\Big)^{\frac{1}{2}}\Big)\Big(\fint_{T(\Delta_{2r})}\tilde{\delta}^{-s\frac{p_0}{p_0-p}}dx\Big)^{\frac{p_0-p}{p_0p}}\Big].\label{eq:longequation_proof}
    \end{align}
    At this point we would like to choose \(s\) such that all integrals in \eqref{eq:longequation_proof} converge. According to \refass{ass:EPS0} we need that
    \begin{align*}
        -\frac{(1-s)q}{p-q}> -1+\EPS_0, \quad 1-s> -1+\EPS_0, \quad\textrm{ and }\quad s< (2-\EPS_0)(1-\frac{p}{p_0}).
    \end{align*}
    Such a choice of \(s\) is possible if and only if 
    \begin{align*}
        1+(1-\EPS_0)(1-\frac{p}{q})< (2-\EPS_0)(1-\frac{p}{p_0}),
    \end{align*}
    which is equivalent to 
    \begin{align}
        -\frac{1}{q}+\EPS_0\frac{1}{q}<-\frac{2}{p_0}+\EPS_0\frac{1}{p_0}.\label{eq:FinalCondEqFors}
    \end{align}
    To satisfy \eqref{eq:FinalCondEqFors} we can choose any \(q<p_0/2\) with a sufficiently small \(\EPS_0\).
    Hence, we obtain \eqref{eq:WTSestimate} for \(q<\min\{p,p_0/2\}\) and a sufficiently small \(\EPS_0=\EPS(q,p_0)>0\), which finishes the proof.
\end{proof}

\begin{rem}\label{rem:DependenciesOfqp0}
    For the previous proof, we would like to also give some more perspective on the dependencies of \(\EPS_0, p_0\) and the solvability exponent \(q\) which is given by \eqref{eq:FinalCondEqFors}. In principle, if \(p< p_0/2\), we can reach up to the same solvability exponent for the mixed boundary value problem as for the pure Neumann boundary value problem for some small \(\EPS_0>0\). Furthermore, we can observe that the smaller \(q\) we have, the larger we can allow \(\EPS_0\) to be. 

    For instance, if \(L=\Delta\) is the Laplacian, then \(p=2+\EPS\) for some small \(\EPS>0\), and if the boundary \(\partial\Omega\) and the interface \(\Lambda\) are Reifenberg flat, we have \(p_0=4-\delta\) for any small \(\delta>0\). In that case, we can obtain the local nontangential maximal function estimate for every \(q\in [1,2)\), where \(\Lambda\) has to be flatter the closer \(q\) gets to \(2\). This was already proved in \cite{dong_dirichlet-conormal_2022}.
\end{rem}

\section{Existence of solutions to the \texorpdfstring{\(L^1\)}{} and \texorpdfstring{\(L^q\)}{} mixed boundary value problem}\label{section:ExistenceProof}

Since the Hardy space is the correct boundary value space for the \(L^1\)-solvability, we recall the following definition.
\begin{defin}[Hardy space]\label{def:HardySpaces}
    We say that \(a\) is Hardy atom on \(\mathcal{N}\) if \(a\) is supported in \(\mathcal{N}\cap \Delta_r\) for some boundary ball \(\Delta_r\) with \(r<R_0/2\) and such that
    \begin{align}\Vert a\Vert_{L^\infty}\leq\frac{1}{|\Delta_r|}, \quad\textrm{ and }\quad \fint_{\Delta_r}ad\sigma=0 \textrm{ when }\Delta_r\subset\mathcal{N}.\label{eq:HardyAtomProperty}\end{align}
    Then we say \(g\) lies in the Hardy space \(\mathcal{H}^1(\mathcal{N})\) if there exists a sequence of Hardy atoms \((a_i)_{i\in \mathbb{N}}\) and real coefficients \((b_i)_{i\in\mathbb{N}}\) such that
    \begin{align}
    g=\sum_{i\in\mathbb{N}}b_ia_i, \textrm{ with the norm }\Vert g\Vert_{\mathcal{H}^1(\mathcal{N})}:=\sum_{i\in\mathbb{N}}|b_i|<\infty,\label{eq:HardyNormAtom}
    \end{align}
    where the first equality is a distributional equality.

    Furthermore, we say \(a\) is a Hardy-Sobolev atom on \(\mathcal{D}\) if \(a\) is supported in \(\mathcal{D}\cap \Delta_r\) for some boundary ball \(\Delta_r\) with \(r<R_0/2\) and such that
    \begin{align}\Vert \nabla a\Vert_{L^\infty}\leq\frac{1}{|\Delta_r|}.\label{eq:HardySobolevAtomProperty}\end{align}

    Similarly, we say \(g\) lies in the Hardy-Sobolev space \(\dot{HS}^1(\mathcal{D})\) if there exists a sequence of Hardy-Sobolev atoms \((a_i)_{i\in \mathbb{N}}\) and real coefficients \((b_i)_{i\in\mathbb{N}}\) such that
    \begin{align}g=\sum_{i\in\mathbb{N}}b_ia_i, \textrm{ with norm }\Vert g\Vert_{\dot{HS}^1(\mathcal{N})}:=\sum_{i\in\mathbb{N}}|b_i|<\infty,\label{eq:HardySobolevNormAtom}\end{align}
    where the first equality is a distributional equality.
\end{defin}

As \(g_\mathcal{D}\in \dot{HS}^1(\mathcal{D})\),
the following lemma and the solvability of the \((R)_1\) allow us to reduce the mixed boundary value problem \eqref{eq:MixedPDE} to \eqref{PDEwithZeroDirichletdata} with boundary data \(g_\mathcal{D}\equiv 0\) and Hardy atom data \(g_\mathcal{N}\in\mathcal{H}^1(\mathcal{N})\).

\begin{lemma}[{\cite[Lemma 4.0.1]{mourgoglou_endpoint_2011}}]
    Suppose that \(u\in W^{1,2}_{\text{loc}}(\Omega)\) solves \(Lu=0\) and \(\tilde{N}(\nabla u)\in L^1(\partial\Omega)\). Then
    \begin{align*}
        \Vert A\nabla u\cdot \nu\Vert_{\mathcal{H}^1_{atom}(\partial\Omega)}\lesssim \Vert\tilde{N}(\nabla u)\Vert_{L^1(\partial\Omega)}.
    \end{align*}
\end{lemma}

Now, assume that \(u\) solves \eqref{PDEwithZeroDirichletdata} with \(g_\mathcal{N}\) a Hardy atom on \(\mathcal{N}\).
Since \(g_\mathcal{N}\in L^\infty\subset W^{-1/2,2}(\partial\Omega)\), according to \refthm{thm:ExistenceofW12sol}, there exists a solution \(u\in W^{1,2}(\Omega)\) to \eqref{PDEwithZeroDirichletdata}.

 Set the annuli \[\Sigma_0:=\Delta_{2r} \quad \textrm{and}\quad\Sigma_k:=\Delta_{2^{k+1}r}\setminus \Delta_{2^{k}r}\quad \textrm{for }1\leq k\leq K. \]
Since our domain \(\Omega\) is bounded, we can choose \(K=K(r)\) as the largest constant such that \(2^Kr<R_0\). Recall that \(R_0\) is the constant from \refass{ass:corkscrew}, and that \(p_0\) is the reverse Hölder exponent from \reflemma{lemma:RevHolderBoundary} and \(p>1\) such that \((R)_p\) and \((N)_p\) are solvable.

\subsection{Decay estimates for the nontangential maximal function}
We have the following decay estimate.
\begin{prop}\label{prop:DecayEstimateClosetruncated}
    Let \(u\) be a solution to \eqref{PDEwithZeroDirichletdata} with the Hardy atom \(g_\mathcal{N}\) supported on \(\Delta_r\) with \(r\leq R_0\). Then there exists \(\beta>0\) such that for all \(q\in (1,\min\{p,p_0/2\})\) we have
    \[\Big(\fint_{\Sigma_k}\tilde{N}_{2^{k-2}r}(|\nabla u|)^qdx\Big)^{1/q}\lesssim 2^{-\beta k}|\Sigma_k|^{-1}\quad \textrm{for all }0\leq k\leq K\]
    and
    \begin{align}\Vert \tilde{N}_{R_0/4}(\nabla u)\Vert_{L^q(\partial\Omega\setminus\Delta_{2^Kr})}\leq C.\label{eq:DecayPropPartLargerR0}\end{align}
\end{prop}

\begin{proof}
    For the local part \(k=0\), we have by \reflemma{lemma:LocalNontangnetialEstimatePureNeumann}, \reflemma{lemma:LocalNontangnetialEstimatePureDirichlet}, or \reflemma{lemma:LocalNontangnetialEstimateMixed},
    \begin{align*}
        \Big(\fint_{\Delta_{2r}}\tilde{N}_{r/4}(|\nabla u|)^qdx\Big)^{1/q}\lesssim \Vert g_\mathcal{N}\Vert_{L^\infty} + \Big(\fint_{T(\Delta_{4r})}|\nabla u|^2dx\Big)^{1/2}.
    \end{align*}
    Due to the embedding of \(L^{2(n-1)/n}\) into \(W^{-1/2,2}\), \refthm{thm:ExistenceofW12sol} yields
    \begin{align*}
        \Vert \nabla u\Vert_{L^2(\Omega)}\lesssim \Vert g_\mathcal{N}\Vert_{L^{2(n-1)/n}(\partial\Omega)}\lesssim \Vert g_\mathcal{N}\Vert_{L^\infty}|\Delta_r|^{\frac{n}{2(n-1)}},
    \end{align*}
    which implies that
    \[\Big(\fint_{T(\Delta_{4r})}|\nabla u|^2dx\Big)^{1/2}\lesssim r^{-\frac{n}{2}}\Vert \nabla u\Vert_{L^2(\Omega)}\lesssim \Vert g_\mathcal{N}\Vert_{L^\infty}.\]
    All together, since \(g_\mathcal{N}\) is a Hardy atom, \eqref{eq:HardyAtomProperty} implies
    \[\Big(\fint_{\Delta_{2r}}\tilde{N}_{r/4}(|\nabla u|)^qdx\Big)^{1/q}\lesssim \Vert g_\mathcal{N}\Vert_{L^\infty}\lesssim |\Delta_r|^{-1}.\]

    For the away annuli, \(\Sigma_k\), \(k\geq 1\), there are \(M\) boundary balls \(\Delta^{(i,k)}=\Delta_{2^{k-2}r}(y_i)\) whose union covers \(\Sigma_k\). Note that \(M\in\mathbb{N}\) is independent of \(k\). Then on each of these \(\Delta^{(i,k)}\) we have
    \begin{align*}
    \Big(\fint_{\Delta^{(i,k)}}\tilde{N}_{2^{k-2}r}(|\nabla u|)^qdx\Big)^{1/q}&\lesssim\fint_{T(2\Delta^{(i,k)})} |\nabla u| dx\\
    &\lesssim (2^kr)^{-n/2}\left(\int_{T(2\Delta^{(i,k)})} |\nabla u|^2 dx\right)^{1/2},    
    \end{align*}
    by \reflemma{lemma:LocalNontangnetialEstimatePureNeumann}, \reflemma{lemma:LocalNontangnetialEstimatePureDirichlet} or \reflemma{lemma:LocalNontangnetialEstimateMixed}.
    
    By duality there exists \(h\in L^2(T(2\Delta^{(i,k)}))\) with \(\Vert h\Vert_{L^2(\Omega)}=1\) such that
    \[\left(\int_{T(2\Delta^{(i,k)})} |\nabla u|^2 dx\right)^{1/2}=\int_\Omega \nabla u\cdot h dx.\]
    Let us choose \(v\) as the solution to
    \begin{align}\begin{cases}
    L^*v=\mathrm{div}(h)& \textrm{ in }\Omega,
    \\
    A^*\nabla v\cdot \nu=0 &\textrm{ on }\mathcal{N},
    \\
    v=0 &\textrm{ on }\mathcal{D}.
    \end{cases}
    \end{align}
    By Lax-Milgram, it is easy to see that \(v\in W^{1,2}(\Omega)\) with \[\Vert \nabla v\Vert_{L^2(\Omega)}\lesssim \Vert h\Vert_{L^2}=1.\]

Now together, we obtain by use of the PDE for \(v\) and the PDE for \(u\)
\begin{align}
    &(2^kr)^{n/2}\Big(\fint_{\Delta^{(i,k)}}\tilde{N}_{2^{k-2}r}(|\nabla u|)^qdx\Big)^{1/q}\lesssim\int_\Omega \nabla u\cdot h dx\label{eq:decayEstimate1}
    \\
    &=\int_\Omega A\nabla u \cdot \nabla v dx=\int_{\Delta_r}g_\mathcal{N} (v-v(x_0))\,d\sigma.\nonumber
\end{align}
Here we choose a point \(x_0\in \Delta_r\cap\mathcal{D}\), if \(\Delta_r\cap \mathcal{D}\neq\emptyset\), and as an arbitrary point in \(\Delta_r\), if \(\Delta_r\cap\mathcal{D}= \emptyset\).

Since \(v\) solves the PDE \(L^*v=0\) in \(T(\Delta_{2^{k-2}r})\) with zero Dirichlet, zero Neumann, or zero mixed data respectively on \(\Delta_{2^{k-2}r}\), we can apply Boundary H\"{o}lder (\refprop{prop:BoundaryHoelder}), Moser estimates (\refprop{prop:MoserEstimate}), and Poincar\'{e} to obtain
    \begin{align}
        \mathrm{osc}_{\Delta_r}v&\lesssim 2^{-\beta k}  \mathrm{osc}_{\Delta_{2^{k-4}r}}v\label{eq:decayEstimate2}
        \lesssim  2^{-\beta k} (2^kr)\fint_{T(\Delta_{2^{k-3}r})}|\nabla v|dx
        \\
        &\lesssim  2^{-\beta k} (2^kr)^{-n/2+1}\Big(\int_{T(\Delta_{2^{k-3}r})}|\nabla v|^2dx\Big)^{1/2}\nonumber
        \\
        &\lesssim 2^{-\beta k} (2^kr)^{-n/2+1}\Vert \nabla v\Vert_{L^2(\Omega)}
        \lesssim 2^{-\beta k} (2^kr)^{-n/2+1}.\nonumber
    \end{align}
Hence, we obtain in total
\begin{align*}
    \Big(\fint_{\Delta_{2r}}\tilde{N}_{r/4}(|\nabla u|)^qdx\Big)^{1/q}\lesssim 2^{-\beta k}(2^kr)^{-n/2}(2^kr)^{-n/2+1}\lesssim 2^{-\beta k}|\Sigma_k|^{-1}.
\end{align*}

An analogous argument as for the away annuli also yields \eqref{eq:DecayPropPartLargerR0}. The only point to be noted here is that \(|\Sigma_K|\approx (2^Kr)^{n-1}\approx R_0^{n-1}\approx 1\), since \(R_0\) is chosen as a fixed constant.
\end{proof}

\begin{prop}\label{prop:DecayEstimateAwaytruncated}
    Let \(u\) be a solution to \eqref{PDEwithZeroDirichletdata} with the Hardy atom \(g_\mathcal{N}\) supported on \(\Delta_r\) with \(r\leq R_0\). Then there exists \(\beta>0\) such that for all \(q\in (1,\min\{p,p_0/2\})\) we have
    \[\Big(\fint_{\Sigma_k}\tilde{N}^{2^{k-2}r}(|\nabla u|)^qdx\Big)^{1/q}\lesssim 2^{-\beta k}|\Sigma_k|^{-1}\quad \textrm{for all }0\leq k\leq K\]
    and
    \begin{align}\Vert \tilde{N}^{R_0/4}(\nabla u)\Vert_{L^q(\partial\Omega\setminus\Delta_{2^Kr})}\lesssim C.\label{eq:DecayPropPartLargerR0awaytruncated}\end{align}
\end{prop}

\begin{proof}
For \(k\geq 0\), we have the pointwise bound
\begin{align*}
    \sup_{x\in\Sigma_k}\tilde{N}^{2^{k-2}r}(\nabla u)(x)&=\sup_{y\in\Gamma(x)\cap\{\delta(y)\geq 2^{k-2}r\}}\left(\fint_{B_{\delta(y)/2}(y)}|\nabla u|^2\right)^{1/2}
    \\
    &\lesssim (2^kr)^{-n/2}\Vert \nabla u\Vert_{L^2(\Omega\setminus T(\Delta_{2^{k-2}r}))}
    \lesssim (2^kr)^{-n/2}\int_\Omega \nabla u \cdot h dy
\end{align*}
for a dualising function \(h\in L^2(\Omega)\) with \(\Vert h\Vert_{L^2}=1\) and \(\supp(h)\subset \Omega\setminus T(\Delta_{2^{k-2}r})\). Now, we can proceed completely analogously to \eqref{eq:decayEstimate1} and \eqref{eq:decayEstimate2}, since all the required results, like the Moser estimates (\refprop{prop:MoserEstimate}), and Boundary H\"{o}lder continuity (\refprop{prop:BoundaryHoelder}), only rely on the fact that the distance of \(\Delta_r\) to the support of \(h\) is comparable to \(2^k\).
Hence we obtain
\[\Big(\fint_{\Sigma_k}\tilde{N}^{2^{k-2}r}(|\nabla u|)^qdx\Big)^{1/q}\lesssim  2^{-\beta k}(2^kr)^{-n/2}(2^kr)^{-n/2+1}\lesssim 2^{-\beta k}|\Sigma_k|^{-1}.\]
The same argument for \(k=K\) also yields \eqref{eq:DecayPropPartLargerR0awaytruncated}, since \(2^Kr\approx R_0\).
\end{proof}

\subsection{Proof of \texorpdfstring{\refthm{MainThm} \eqref{item:MainThmA} and  \eqref{item:MainThmC}}{}}

With the decay estimates on both, the away truncated and the close truncated, versions of the nontangential maximal function, we can, finally, obtain the \(L^1\) solvability of the mixed boundary value problem \eqref{PDEwithZeroDirichletdata}.

\begin{proof}[Proof of \refthm{MainThm} \eqref{item:MainThmA}]
Let \(g_\mathcal{N}\) be a Hardy atom supported on \(\Delta_r\). Choose \(q>1\) to be such that \reflemma{lemma:LocalNontangnetialEstimateMixed} applies. Then, we have
\begin{align*}
    \int_{\partial\Omega}\tilde{N}(\nabla u) d\sigma&=\sum_{k=0}^K\Big(\int_{\Sigma_k}\tilde{N}_{2^{k-2}r}(\nabla u) d\sigma + \int_{\Sigma_k}\tilde{N}^{2^{k-2}r}(\nabla u) d\sigma \Big)
    \\
    &\qquad + \Big(\int_{\partial\Omega\setminus\Delta_{R_0/2}}\tilde{N}_{R_0/4}(\nabla u) d\sigma + \int_{\partial\Omega\setminus\Delta_{R_0/2} }\tilde{N}^{R_0/4}(\nabla u) d\sigma \Big)
    \\
    &\lesssim\sum_{k=0}^K|\Sigma_k|\Big(\big(\fint_{\Sigma_k}\tilde{N}_{2^{k-2}r}(\nabla u)^q d\sigma\big)^{1/q} + \big(\fint_{\Sigma_k}\tilde{N}^{2^{k-2}r}(\nabla u)^q d\sigma\big)^{1/q} \Big)
    \\
    &\qquad + \big(\int_{\partial\Omega\setminus\Delta_{R_0/2}}\tilde{N}_{R_0/4}(\nabla u)^q d\sigma\big)^{1/q} + \big(\int_{\partial\Omega\setminus\Delta_{R_0/2} }\tilde{N}^{R_0/4}(\nabla u)^q d\sigma \big)^{1/q}
    \\
    &\lesssim \sum_{k=0}^K 2^{-\beta k}C + C\leq C.
\end{align*}
This finishes the proof.
\end{proof}

\begin{proof}[Proof of \refthm{MainThm} \eqref{item:MainThmC}]
    The proof of extending \(L^1\) solvability to \(L^q\) solvability follows the idea of \cite{dong_dirichlet-conormal_2022} (or also \cite{dong_dirichlet-conormal_2020}). We want to apply the following interpolation result:
    \begin{prop}
        Let \(\Delta_0\subset\partial\Omega\) be a boundary cube such that \(\mathrm{diam}(4\Delta_0)<R_0\), \(s,q\) be real numbers satisfying \(1<s<q\), and \(F,g\) two functions defined on \(4\Delta_0\). Suppose there exists \(C_0>0\) that for any surface cube \(\Delta\subset \Delta_0\), we can find \(F_\Delta\) and \(R_\Delta\) defined on \(\Delta\) such that
        \begin{enumerate}
            \item \(|F|\leq C_0(|F_\Delta|+|R_\Delta|)\quad \textrm{ on }\Delta\),
            \item \(\left(\fint_\Delta |R_\Delta|^q\right)^{1/q}\leq C_0\left(\fint_{4\Delta}|F| + \sup_{\Delta'\supset2\Delta}\big(\fint_{\Delta'}|g|^s\big)^{1/s}\right)\), and
            \item \(\fint_\Delta |F_\Delta|\leq C_0\sup_{\Delta'\supset2\Delta}\big(\fint_{\Delta'}|g|^s\big)^{1/s}\).
        \end{enumerate}
        Then for any \(\tilde{q}\in (s,q)\), if \(g\in L^{\tilde{q}}(4\Delta_0)\), we also have \(F\in L^{\tilde{q}}(\Delta_0)\) with
        \[\left(\fint_{\Delta_0}|F|^{\tilde{q}}\right)^{1/\tilde{q}}\leq C\left(\fint_{4\Delta_0}|F| + \big(\fint_{4\Delta_0}|g|^{\tilde{q}} \big)^{1/\tilde{q}}\right),\]
        where \(C=C(n,C_0,s,q,p)>0\) is some constant.
    \end{prop}
    As in section 6 of \cite{dong_dirichlet-conormal_2022} we choose
    \[F:=M[\tilde{N}(\nabla u)^{1/2}]^{2}, \quad F_\Delta:=M[\tilde{N}(\nabla w)^{1/2}]^{2}, \textrm{ and } \quad R_\Delta:=M[\tilde{N}(\nabla v)^{1/2}]^{2},\]
    where \(M\) is the Hardy-Littlewood maximal function, \(u\) the weak solution to \eqref{PDEwithZeroDirichletdata} with \(g_\mathcal{N}\in L^q(\mathcal{N}), 1<q<\min\{p,p_0/2\}\), and 
    \[\begin{cases}
    Lw=0& \textrm{ in }\Omega,
    \\
    A\nabla w\cdot \nu=g_\mathcal{N}1_{2\Delta} &\textrm{ on }\mathcal{N},
    \\
    w=0 &\textrm{ on }\mathcal{D},
    \\
    \tilde{N}(\nabla w)\in L^1(\partial\Omega)
\end{cases}\quad \textrm{and}\quad v:=u-w.\]

For the rest of the argument, we can follow the proof of Proposition 6.2 in \cite{dong_dirichlet-conormal_2022} almost verbatim, since the only results of the PDE that are used are \reflemma{lemma:LocalNontangnetialEstimatePureNeumann}, \reflemma{lemma:LocalNontangnetialEstimatePureDirichlet}, and \reflemma{lemma:LocalNontangnetialEstimateMixed}. To be able to apply \reflemma{lemma:LocalNontangnetialEstimateMixed}, we have to choose \(q\) between \(1\) and \(\min\{p,p_0/2\}\) and \(\EPS_0\) for \refass{ass:EPS0} sufficiently small as stated in \reflemma{lemma:LocalNontangnetialEstimateMixed}.

The only adaptation in contrast to the proof of Proposition 6.2 in \cite{dong_dirichlet-conormal_2022}, that we have to make, is, that all nontangential maximal function are replaced with mean-valued nontangential maximal functions. In particular in \((6.9)\) in \cite{dong_dirichlet-conormal_2022}, we have to replace the pointwise bound for \(y\in \Gamma_\alpha(x)\) with
\[\big(\fint_{B_{\delta(y)/2}(y)}|\nabla v|^2\big)^{1/4}\leq \fint_{\Delta_{c\alpha}(\tilde{y})}\tilde{N}(\nabla v)^{1/2},\]
where \(\tilde{y}\in \partial\Omega\) with \(|\tilde{y}-y|=\delta(y)\), which follows by the same arguments as in \cite{dong_dirichlet-conormal_2022}.
\end{proof}

\begin{rem}[Proof of \refcor{MainCorForSmallDKP} \eqref{item:MainCorA} and \eqref{item:MainCorC}]
    We would like to point out that for the proof of \refcor{MainCorForSmallDKP} \eqref{item:MainCorA} and \eqref{item:MainCorC}, we only use the small DPR condition, \refass{ass:smallDPR}, to apply \refprop{prop:SolvabilityofDRNforDKP} and conclude solvability of \((R)_p\) and \((N)_p\). Then \refthm{MainThm} \eqref{item:MainThmA} and \eqref{item:MainThmC} apply.
\end{rem}

\section{Uniqueness - Proof of \texorpdfstring{\refthm{MainThm} \eqref{item:MainThmB}}{}}\label{section:UniquenessProof}

Let \(u\in W^{1,2}_{\text{loc}}(\Omega)\) be a solution to
\begin{align}\begin{cases}
    Lu=0& \textrm{ in }\Omega,
    \\
    A\nabla u\cdot \nu=0 &\textrm{ on }\mathcal{N},
    \\
    u=0 &\textrm{ on }\mathcal{D},
    \\
    \tilde{N}_\alpha(\nabla u)\in L^1(\partial\Omega).
\end{cases}\label{eq:MixedPDEZeroData}
\end{align}

To prove uniqueness of a solution to the \(L^1\) mixed boundary value problem, it suffices to prove that the only solution \(u\) to \eqref{eq:MixedPDEZeroData} is \(u\equiv 0\). For that, we are going to show that
\begin{align}\int_Qu(x)dx=0\quad \textrm{ for all cubes }Q\subset \Omega.\label{eq:WTSIntuoverQ}\end{align}
First, let us fix a cube \(Q\subset\Omega\). By Lax-Milgram, let \(v\in W^{1,2}(\Omega)\) be the solution to 
\begin{align}\begin{cases}
    L^*v=1_Q& \textrm{ in }\Omega,
    \\
    A^*\nabla v\cdot \nu=0 &\textrm{ on }\mathcal{N},
    \\
    v=0 &\textrm{ on }\mathcal{D},
\end{cases}\label{eq:PDEforv}
\end{align}
where \(1_Q\) is the characteristic function on \(Q\). Since \(u\) is merely in \(W^{1,2}_{\text{loc}}\) apriori, the regularity of \(u\) is not sufficient for integration against the Poisson data of \eqref{eq:PDEforv} to use integration by parts methods. Thus, we regularize \(u\) by the small parameter \(t>0\). For that, we introduce \(\eta=\eta_t\in C^\infty(\Omega)\) such that
\begin{enumerate}[(a)]
    \item \(\eta\equiv 1\) on \(\Omega\setminus \tilde{\Omega}_{\mathcal{D},\alpha/2}^{2t}\), where \(\tilde{\Omega}_{\mathcal{D},\alpha/2}^{2t}:=\bigcup_{x\in \mathcal{D}}\Gamma^{2t}_{\alpha/2}(x)\) is the sawtooth region over \(\mathcal{D}\) truncated at height \(t\);
    \item \(\eta\equiv 0\) on \(\tilde{\Omega}_{\mathcal{D},\alpha/4}^t\);
    \item we set \(\Lambda^t:=\Omega\cap\{x\in\mathrm{supp}(\eta):\tilde{\delta}(x)< 3t\}\) and for \(x\in \supp(\eta)\):\label{enumerate:DerivativeOfEta}
    \begin{itemize}
        \item \(|\nabla \eta(x)|\lesssim \frac{1}{t}\) if \(x\notin\Lambda^t\)
        \item \(\delta(x)\approx \tilde{\delta}(x)\) and \(|\nabla \eta(x)|\lesssim \tilde{\delta}(x)^{-1}\), if \(x\in\Lambda^t\).
    \end{itemize}
\end{enumerate}
From above definition it follows that \(\eta\equiv 0\) on \(\mathcal{D}\) and \(\eta\equiv 1\) on \(\mathcal{N}\), and that \(\eta_h\uparrow 1_\Omega\) pointwise for \(h\to 0\). This special form of the cut-off \(\eta\) is motivated by Lemma 5.1 in \cite{dong_dirichlet-conormal_2022}.
\medskip

Now, we introduce a partition of unity by a family of smooth cut-off functions \(0\leq \chi_k\leq 1\) such that 
\begin{itemize}
    \item \(1_\Omega\equiv \sum_{k=0}^N\chi_k\);
    \item \(\mathrm{dist}(\mathrm{supp}(\chi_0), \partial\Omega)>0\) and \(\chi_0|_{(1+\rho)Q}\equiv 1\) for some small enlargement of \(Q\) given by \(\rho>0\);
    \item and for \(k=1,...,N\) there is a Lipschitz function \(\psi_k\) such that \(\mathrm{supp}(\chi_k)\) is equal to \(\{x_n>\psi_k(x_1,..,x_{n-1})\}\) after rotation.
\end{itemize}
In the following, for a point \(x\in \mathrm{supp}(\chi_k)\), \(k\geq 1\), let us denote the projection of the point onto the boundary as \(\tilde{x}:=(x_1,...,x_{n-1}, \psi_k(x_1,...,x_{n-1}))\). Since there are only finitely many \(\chi_k\), we also have \(|\nabla \chi_k|\lesssim 1\).

Now, let us fix one of these cut-off functions, i.e., \(k\geq 1\). For small \(h>0\), in the coordinates of \(\supp(\chi_k)\), we can define a translation away from the boundary by
\[\tau^{(k)}_h(x)=(x_1,...,x_{n-1},x_n+h).\]
Finally, we would like to consider \(w\) as the regularization of \(u\) by  \(h\) on \(\mathrm{supp}(\chi_k)\), where \(w\) is given by
\[w:=w_h^{(k)}:= \fint_h^{2h}\big((u\chi_k)\circ\tau_t^{(k)}\big)\eta_t\,dt\in W^{1,2}(\Omega).\]

We note that \(w|_\mathcal{D}=0\). Now, 
using  \eqref{eq:PDEforv},
\begin{align}
    \int_Q w_h^{(k)}1_Q dx=-\int_\Omega A\nabla w^{(k)}_h\cdot \nabla v dx + \int_{\partial\Omega}w_h^kA^*\nabla v\cdot \nu dx.\label{eq:WTSwithWinsteadofu}
\end{align}
Since \(A^*\nabla v\cdot \nu=0\) on \(\mathcal{N}\) and \(w^{(k)}_h=0\) on \(\mathcal{D}\), the boundary term vanishes, and we can continue with
\begin{align}
    &= -\int_\Omega\fint_h^{2h} A\nabla ((\chi_k\circ\tau^{(k)}_t) \eta_t) \cdot\nabla v\,(u\circ\tau_t^{(k)}) + A\nabla (u\circ\tau_t^{(k)}) \cdot\nabla v\,(\chi_k\circ\tau^{(k)}_t) \eta_t \,dtdx\nonumber
    \\
    &= -\int_\Omega\fint_h^{2h} A\nabla ((\chi_k\circ\tau^{(k)}_t) \eta_t) \cdot\nabla v\,(u\circ\tau_t^{(k)})  +(A-A\circ\tau_t^{(k)})\nabla (u\circ\tau_t^{(k)}) \cdot\nabla v\,(\chi_k\circ\tau^{(k)}_t) \eta_t\,dtdx\nonumber
    \\
    &\qquad  -\int_\Omega\fint_h^{2h} A\circ\tau_t^{(k)}\nabla (u\circ\tau_t^{(k)}) \cdot\nabla v\,(\chi_k\circ\tau^{(k)}_t) \eta_t \,dtdx\label{eq:Eqforfg_kWithoutNablav}
    \\
    &=
    -\fint_h^{2h}\int_\Omega F_t^{(k)} \cdot\nabla v \,dx\,dt + \fint_h^{2h}\int_\Omega f_t^{(k)} v \,dx\,dt\nonumber
    \\
    &\qquad - \int_\mathcal{N}\fint_h^{2h}A\circ\tau_t^{(k)}\nabla (u\circ\tau_t^{(k)})\cdot \nu \,(\chi_k\circ\tau^{(k)}_t)\eta_t v\,dt d\sigma \nonumber
    \\
    &=:\fint_{h}^{2h}(-{\rm I}_t^{(k)}+{\rm II}_t^{(k)})dt-{\rm III}_h^{(k)},\nonumber
        \end{align}
where we used integration by parts on the term in \eqref{eq:Eqforfg_kWithoutNablav} to move the gradient away from \(v\). Hence, the functions \(F_t^{(k)}\) and \(f_t^{(k)}\) are given by
\begin{align*}F^{(k)}_t:=A\nabla (\chi_k\circ\tau_t^{(k)}) \eta_t (u\circ\tau_t^{(k)})
+ A\nabla \eta_t (\chi_k\circ\tau_t^{(k)})(u\circ\tau_t^{(k)})
\\
+(A-A\circ\tau_t^{(k)})\nabla(u\circ\tau_t^{(k)}) (\chi_k\circ\tau_t^{(k)})\eta_t
\end{align*}
and
\begin{align*}
    f_t^{(k)}:=(A\circ\tau_t^{(k)})\nabla (u\circ\tau_t^{(k)})\cdot \nabla(\chi_k\circ\tau^{(k)}_t) \eta_t + (A\circ\tau_t^{(k)})\nabla (u\circ\tau_t^{(k)})\cdot \nabla\eta_t(\chi_k\circ\tau^{(k)}_t)
    \\
    + \mathrm{div}((A\circ\tau_t^{(k)})\nabla(u\circ\tau_t^{(k)})) (\chi_k\circ \tau_t^{(k)})\eta_t. 
\end{align*}

Now, we claim that for \(h\to 0\) we get
\begin{align}
    {\rm I}={\rm I}^{(k)}_h&\to \int_\Omega A\nabla\chi_k\cdot \nabla v u \,dx,\label{eq:I}
    \\
    {\rm II}={\rm II}^{(k)}_h&\to \int_\Omega A\nabla u\cdot \nabla\chi_k v \,dx,\label{eq:II}
    \\
    {\rm III}={\rm III}^{(k)}_{h}&\to 0\label{eq:III}
\end{align}
We postpone the proof of \eqref{eq:I}, \eqref{eq:II}, and \eqref{eq:III} and show first how the proof of \eqref{eq:WTSIntuoverQ} is completed.

For that, we show first by the dominated convergence theorem that
\begin{align}\int_Q u\chi_k dx=\lim_{h\to 0}\int_Q w^{(k)}_hdx\quad\textrm{  for every }k\geq 1.\label{eq:IntOveruConvtoIntOverw}\end{align}

We note that for \(x\in \supp(\chi_k)\) and by the fundamental theorem of calculus we have
\begin{align}
    u\circ\tau_h^{(k)}(x)\lesssim (\delta(x)+h)\tilde{N}(\nabla u)(\tilde{x})+u(\tilde{x}).\label{eq:UpointwiseBoudnByN(nablau)}
\end{align} 
With this at hand, we can bound 
\[w_h^{(k)}(x)=\fint_h^{2h}(u\chi_k)\circ\tau_t^{(k)}\eta_t dt\lesssim (\delta(x)+h)\tilde{N}(\nabla u)(\tilde{x}) +u(\tilde{x})\]
pointwise, and note that \((\delta(x)+h)\tilde{N}(\nabla u)(\tilde{x}) +u(\tilde{x})\) is integrable. This integrability follows by the Poincar\'{e} inequality and the assumption \(\Vert\tilde{N}(\nabla u)\Vert_{L^1(\partial\Omega)}<\infty\), since
\begin{align*}
    &\int_{\supp(\chi_k)}(\delta(x)+h)\tilde{N}(\nabla u)(\tilde{x}) +u(\tilde{x})dx\lesssim \Vert \tilde{N}(\nabla u)\Vert_{L^1(\partial\Omega)} +\Vert u\Vert_{L^1(\partial\Omega)}
    \\
    &\qquad\lesssim \Vert \tilde{N}(\nabla u)\Vert_{L^1(\partial\Omega)} +\Vert \nabla u\Vert_{L^1(\mathcal{N})}\lesssim \Vert \tilde{N}(\nabla u)\Vert_{L^1(\partial\Omega)}<\infty.
\end{align*}
This finishes the proof of \eqref{eq:IntOveruConvtoIntOverw}. For \(k=0\), we note that
the support of \(\chi_0\) is a compact subset of \(\Omega\), which yields by integration by parts and \eqref{eq:PDEforv} that
\begin{align*}
\int_\Omega u\chi_0 1_Qdx=\int_\Omega A\nabla u\cdot \nabla\chi_0 v {\color{blue}-} A\nabla\chi_0\cdot \nabla v u \,dx.
\end{align*}
Thus, we have with \eqref{eq:I}, \eqref{eq:II}, \eqref{eq:III}, and \eqref{eq:IntOveruConvtoIntOverw}
\begin{align*}
    \int_\Omega u1_Qdx&=\int_Q u\chi_0dx +\sum_{k=1}^N\int_Q u\chi_kdx=\int_Qu\chi_0dx +\lim_{h\to 0}\sum_{k=1}^N\int_Q w^{(k)}_hdx
    \\
    &=\sum_{k=0}^N\int_\Omega A\nabla u\cdot \nabla\chi_k v - A\nabla\chi_k\cdot \nabla v u \,dx=0.
\end{align*}
The last equality used that \(\nabla \big(\sum_{k=0}^N\chi_k)=\nabla (1_\Omega)=0\), and hence it only remains to prove \eqref{eq:I}, \eqref{eq:II}, and \eqref{eq:II} in the following three steps to complete the proof.

\medskip

Before we turn to the proof of \eqref{eq:I}, we would like to make an improvement of observation \eqref{eq:UpointwiseBoudnByN(nablau)}. By the fundamental theorem of calculus, we can also see that there exists a \(c>0\) such that for all \(\tilde{z}\in \Delta(\tilde{x},c\delta(x))\)
\begin{align}
    u\circ\tau_h^{(k)}(x)\lesssim (\delta(x)+h)\tilde{N}(\nabla u)(\tilde{z})+u(\tilde{z}),\label{eq:UpointwiseBoudnByN(nablau)2}
\end{align} 
and hence by the Cacciopolli inequality, we even obtain
\begin{align}
    \Big(\fint_{B_{\delta(x)/2}(x)}|\nabla u\circ\tau_h^{(k)}|^2 dy\Big)^{1/2}&\lesssim \frac{1}{\delta(x)}\Big(\fint_{B_{3\delta(x)/4}(x)}|(u\circ\tau_h^{(k)})(y)-u(\tilde{z})|^2 dy\Big)^{1/2}\nonumber
    \\
    &\lesssim \frac{\delta(x)+h}{\delta(x)}\tilde{N}(\nabla u)(\tilde{z})\lesssim \tilde{N}(\nabla u)(\tilde{z}),\label{eq:NablaUpointwiseBoudnByN(nablau)2}
\end{align}
for \(\tilde{z}\in\Delta(\tilde{x},c\delta(x))\) and if \(\delta(x)\geq h\) and \(B_{3\delta(x)/4}(x)\subset\supp(\chi_k)\).
\smallskip

Furthermore, we will also need a Whitney decomposition of the domain \(\Omega\) at multiple steps throughout the rest of the proof. Let \((W_i)_i\) be this Whitney cube collection with the following properties: First, \(l(W_i)\approx\mathrm{dist}(W_i,\partial\Omega)\). Second, we require the union of the \(W_i\) to cover \(\Omega\) in such a way that the enlarged cubes \(\frac{3}{2}W_i\) have finite overlap, i.e., there exists a fixed \(N\in\mathbb{N}\) such that \(\chi_\Omega\leq \sum_i\chi_{W_i}\leq \sum_i\chi_{\frac{3}{2}W_i}\leq N\chi_\Omega\). Without loss of generality, we can also assume that each Whitney cube is chosen such that for all \(y\in W_i\), we have \(\frac{3}{2}W_i\subset B_{\delta(y)/2}(y)\). Now, we are in the position to prove \eqref{eq:I},  \eqref{eq:II}, and  \eqref{eq:III}.
\medskip

\underline{\textbf{Step 1: Proof of \eqref{eq:I}:}}
To begin with, we split
\begin{align*}
    {\rm I}&=\int_\Omega F_t^{(k)}\cdot \nabla v dx=\int_\Omega (A-A\circ\tau_t^{(k)})\nabla(u\circ\tau_t^{(k)}) \cdot \nabla v (\chi_k\circ\tau_t^{(k)})\eta_t dx
    \\
    &\qquad + \int_\Omega A\nabla (\chi_k\circ\tau_t^{(k)})\cdot \nabla v \eta_t (u\circ\tau_t^{(k)})dx
    + \int_\Omega A\nabla \eta_t \cdot\nabla v(\chi_k\circ\tau_t^{(k)})(u\circ\tau_t^{(k)}) dx
    \\
    &=:{\rm I}_1+{\rm I}_2+{\rm I}_3.
\end{align*}

Next, we show that \({\rm I}_1\to 0\) as \(t\to 0\). First, we note that \(\mathrm{supp}(\chi_k\circ\tau_t^{(k)})\subset\mathrm{supp}(\chi_k)\). Next, from Lemma 3.14 in \cite{milakis_harmonic_2013} we obtain
\begin{align*}
    |{\rm I}_1|&\leq \int_{\mathrm{supp}(\chi_k\circ\tau_t^{(k)})}  |A-A\circ\tau_t^{(k)}||\nabla(u\circ\tau_t^{(k)})| |\nabla v| dx
    \\
    &\lesssim \int_{\partial\Omega}\Big(\int_{\Gamma(x)\cap \mathrm{supp}(\chi_k)} |A-A\circ\tau_t^{(k)}| |\nabla(u\circ\tau_t^{(k)})| |\nabla v| \delta^{1-n}dy\Big)d\sigma(x).
\end{align*}
Assuming that \(t\) is small, we fix \(x\in \partial\Omega\) and let us call \(i_0,i_1\in\mathbb{Z}\) the smallest or largest integer, respectively, such that \(2^{-i_0}\leq t\leq \mathrm{diam}(\supp(\chi_k))\leq 2^{-i_1}\).
Then we obtain by Hölder's inequality and the fundamental theorem of calculus for each cone
\begin{align}
    &\int_{\Gamma(x)\cap \mathrm{supp}(\chi_k)}|A-A\circ\tau_t^{(k)}| |\nabla(u\circ\tau_t^{(k)})| |\nabla v| \delta^{1-n}dy\nonumber
    \\
    &\lesssim \sum_{i:W_i\cap\Gamma(x)\cap \mathrm{supp}(\chi_k)\neq \emptyset}l(W_i)t\fint_{W_i}|\partial_{x_n} A||\nabla(u\circ\tau_t^{(k)})| |\nabla v| dy\label{eq:UniquenessDKPneededInsteadOfDPR}
    \\
    &\leq  \sum_{i:W_i\cap\Gamma(x)\cap \mathrm{supp}(\chi_k)\neq \emptyset}l(W_i)t(\mathrm{sup}_{y\in W_i}|\nabla A|)\fint_{W_i}|\nabla(u\circ\tau_t^{(k)})| |\nabla v| dy\nonumber
    \\
    &\leq \sum_{i:W_i\cap\Gamma(x)\cap \mathrm{supp}(\chi_k)\neq \emptyset}t\Big(\frac{1}{l(W_i)^{n-1}}\int_{W_i}\mathrm{sup}_{B_{\delta(y)/2}(y)}|\nabla A|^2\delta dy \Big)^{1/2}\nonumber
    \\
    &\hspace{50mm}\cdot\Big(\fint_{W_i} |\nabla v|^2dy\Big)^{1/2}\Big(\fint_{W_i}|\nabla(u\circ\tau_t^{(k)})|^2dy\Big)^{1/2}.\nonumber
\end{align}
We can continue with the Cacciopolli inequality, the DKP condition \eqref{eq:largeDKPcondition}, and \eqref{eq:NablaUpointwiseBoudnByN(nablau)2} to get
\begin{align*}
    &\lesssim \sum_{i:W_i\cap\Gamma(x)\cap \mathrm{supp}(\chi_k)\neq \emptyset}t\Big(\fint_{\frac{3}{2}W_i} |v-v(x)|^2 \delta^{-2}dy\Big)^{1/2}\tilde{N}(\nabla u)(x)
    \\
    &\lesssim \sum_{i:W_i\cap\Gamma(x)\cap \mathrm{supp}(\chi_k)\neq \emptyset}tl(W_i)^{-1}\Big(\fint_{\frac{3}{2}W_i} |v-v(x)|^2 dy\Big)^{1/2}\tilde{N}(\nabla u)(x).
\end{align*}
By Boundary Hölder (\refprop{prop:BoundaryHoelder}) we obtain \(|v(y)-v(x)|\lesssim \delta(y)^\beta\approx l(W_i)^{\beta}\) for every \(y\in W_i\), and hence integrating in \(x\) and using the finite overlap of the Whitney cubes gives
\begin{align*}
    |{\rm I}_1|&\lesssim t\int_{\partial\Omega}\sum_{i:W_i\cap\Gamma(x)\cap \mathrm{supp}(\chi_k)\neq \emptyset} l(W_i)^{-1+\beta}\tilde{N}(\nabla u) d\sigma(x) 
    \\
    &\lesssim t\int_{\partial\Omega}\sum_{i=i_1}^{i_0}2^{(1-\beta) i}\tilde{N}(\nabla u) d\sigma
    \lesssim t^\beta\Vert \tilde{N}(\nabla u)\Vert_{L^1(\partial\Omega)}.
\end{align*}
Here, we used the geometric sum formula for \(\sum_{i=i_1}^{i_0}2^{(1-\beta) i}
\approx 2^{(1-\beta)i_0}\approx t^{-1+\beta}\). This bound of \({\rm I}_1\) converges to \(0\) as \(t\) tends to \(0\).

Next, we show \({\rm I}_3\to 0\) as \(t\to 0\). Recall that \(\Lambda^t:=\Omega\cap\{x\in\mathrm{supp}(\eta_t):\tilde{\delta}(x)< 3t\}\subset\mathrm{supp}(\eta_t)\) is the part of the support of \(\nabla \eta_t\) that is close to the interface. Then, we can split the integral \({\rm I}_3\) into
\begin{align*}
    |{\rm I}_3|&\leq \int_{\mathrm{supp}(\chi_k)\cap\supp(\nabla \eta_t)\setminus \Lambda^t} t^{-1}| \nabla v| |u\circ\tau_t^{(k)}| dx
    \\
    &\qquad+ \int_{\mathrm{supp}(\chi_k)\cap\Lambda^t} \tilde{\delta}^{-1}|\nabla v| |u\circ\tau_t^{(k)}| dx=:{\rm I}_{31}+{\rm I}_{32}.
\end{align*}
For the first term, \({\rm I}_{31}\), we cover \(\supp(\chi_k)\cap\supp(\nabla \eta_t)\setminus \Lambda^t\) by Whitney cubes, whose size is comparable to \(t\) via definition. Hence we obtain by Hölder's inequality
\begin{align*}
    {{\rm I}}_{31}\leq \sum_{i:W_i\cap(\supp(\chi_k)\cap\supp(\nabla \eta_t)\setminus \Lambda^t)\neq \emptyset} t^{-1}\big(\int_{W_i} |\nabla v|^2dy\big)^{1/2} \big(\int_{W_i} |u\circ\tau_t^{(k)}|^2dy\big)^{1/2}.
\end{align*}
For each cube \(W_i\) with \(W_i\cap\supp(\chi_k)\neq \emptyset\), let us define its projection onto the boundary by \(\tilde{W_i}:=\{z\in \partial\Omega; \exists x\in W_i:\tilde{x}=z\}\). Since for the cubes \(W_i\) covering \(\supp(\chi_k)\cap\supp(\eta_t)\setminus \Lambda^t\), we have \(\tilde{W}_i\subset\mathcal{D}\), it follows that \(u(\tilde{y})=0\) for \(y\in \mathcal{Q}\). Hence, we obtain by \eqref{eq:UpointwiseBoudnByN(nablau)2}, the Cacciopolli inequality and Boundary Hölder (\refprop{prop:BoundaryHoelder})
\begin{align*}
    {\rm I}_{31}&\leq \sum_{i:W_i\cap(\supp(\chi_k)\cap\supp(\nabla \eta_t)\setminus \Lambda^t)\neq \emptyset} t^{-2}\big(\int_{\frac{3}{2}W_i} | v|^2dy\big)^{1/2} \Big(t^{-n/2+2}\int_{\tilde{W}_i} |\tilde{N}(\nabla u)|d\sigma\Big)
    \\
    &\lesssim t^{\beta}\sum_{i:W_i\cap(\supp(\chi_k)\cap\supp(\nabla \eta_t)\setminus \Lambda^t)\neq \emptyset} \int_{\tilde{W}_i} |\tilde{N}(\nabla u)|d\sigma \lesssim t^\beta\Vert\tilde{N}(\nabla u)\Vert_{L^1(\partial\Omega)}.
\end{align*}
Hence \({\rm I}_{31}\) converges to \(0\) if \(t\) goes to \(0\). 

For \({\rm I}_{32}\) we note that, we have \(\tilde{\delta}\approx \delta\) on \(\Lambda^t\cap\supp(\chi_k)\) according to \eqref{enumerate:DerivativeOfEta}. Furthermore, we can find a finite amount of points \((X_j)_j\subset\Lambda\) in the interface with \(|X_{j_1}-X_{j_2}|\geq t/2\), if \(j_1\neq j_2\), such that \(\bigcup_j T(\Delta_{3t}(X_j))\) covers  \(\Lambda^t\cap \supp(\chi_k)\). To abbreviate notation let us denote by \((W_i^j)_i\) a sub collection of the collection of all Whitney cubes with the property that \(W_i^j\cap T(\Delta_{3t}(X_j))\cap(\Lambda^t\cap\supp(\chi_k))\neq \emptyset\). Due to \refass{ass:corkscrew} and the choice of the aperture in the definition of \(\eta_t\), we can choose \(\tilde{x}_j\in \mathcal{D}\) and a constant \(c>0\) such that \(\Delta_{ct}(\tilde{x}_j)\subset \mathcal{D}\cap\Delta_{3t}(X_j)\) and \(\tau_t^{(k)}(T(\Delta_{3t}(X_j))\cap\supp(\nabla\eta_t))\subset \Gamma(\tilde{y})\) for every \(\tilde{y}\in \Delta_{ct}(\tilde{x}_j)\). Hence, we can improve the pointwise bound \eqref{eq:UpointwiseBoudnByN(nablau)2} to the bound \begin{align}\big(\fint_{W_i^j} |u\circ\tau_t^{(k)}|^2dy\big)^{1/2}\lesssim t\fint_{\Delta_{c t}(\tilde{x}_j)}\tilde{N}(\nabla u) dy\label{eq:UpointwiseBoudnByN(nablau)3},\end{align}
where the implicit aperture of the nontangential maximal function is \(\alpha\).

With these preliminary observations at hand, we have by Cacciopolli and Boundary H\"{o}lder (\refprop{prop:BoundaryHoelder})
\begin{align*}
    |{\rm I}_{32}|&\lesssim\sum_j\int_{T(\Delta_{3t}(X_j))\cap\supp(\nabla \eta_t)}\delta^{-1}|\nabla v| |u\circ\tau_t^{(k)}|dy
    \\
    &\lesssim\sum_j\sum_i|W_i^j|\big(\fint_{W^j_i}\delta^{-2}|\nabla v|^2 dy\big)^{1/2} \big(\fint_{W^j_i}|u\circ\tau_t^{(k)}|^2 dy\big)^{1/2}
    \\
    &\lesssim\sum_j\sum_i|W_i^j|\big(\fint_{\frac{3}{2}W^j_i}\delta^{-4}|v-(v)_{W^j_i}|^2 dy\big)^{1/2} \big(t\fint_{\Delta_{c t}(\tilde{x}_j)}\tilde{N}(\nabla u) dy\big)
    \\
    &\lesssim\sum_j \Big(\sum_i \int_{\frac{3}{2}W^j_i}\tilde{\delta}^{-2+\beta}(y)dy\Big)\Big(t^{-n+2}\int_{\Delta_{c t}(\tilde{x}_j)}\tilde{N}(\nabla u) dy\Big)
    \\
    &\lesssim\sum_j \Big(\int_{T(\Delta_{5t}(X_j))}\tilde{\delta}^{-2+\beta}(y)dy\Big)\Big(t^{-n+2}\int_{\Delta_{c t}(\tilde{x}_j)}\tilde{N}(\nabla u) dy\Big) .
\end{align*}
By \refass{ass:EPS0} and if \(\EPS_0\leq \beta\), we can continue with
\begin{align*}
    \lesssim t^{-n+2}\Vert\tilde{N}(\nabla u)\Vert_{L^1(\partial\Omega)}t^{n-2+\beta}=t^\beta\Vert\tilde{N}(\nabla u)\Vert_{L^1(\partial\Omega)}.
\end{align*}
Hence, we get that \(|{\rm I}_{32}|\to 0\) as \(t\to 0\).

Lastly, for \({\rm I}_2\), we want to use the dominated convergence theorem, since it is clear that the integrand of \({\rm I}_2\) converges pointwise almost everywhere to the integrand of the right-hand side of \eqref{eq:I}. For that, we note that \(|\nabla \chi_k\circ\tau_t^k|\lesssim 1\). Furthermore, by \eqref{eq:UpointwiseBoudnByN(nablau)2}, if \(W_i\cap\mathrm{supp}(\chi_k)\neq \emptyset\), then we have 
\[u\circ\tau_t^{(k)}(x)\lesssim\inf_{z\in W_i} (\delta(x)+t)\tilde{N}(\nabla u)(\tilde{z})+u(\tilde{z})=\inf_{\tilde{z}\in \tilde{W}_i} (\delta(x)+t)\tilde{N}(\nabla u)(\tilde{z})+u(\tilde{z})\]
for every \(x\in W_i\). Hence the integrand of \(I_2\) is pointwise bounded by the function
\[x\mapsto \sum_{i:W_i\cap\mathrm{supp}(\chi_k)\neq \emptyset}C|\nabla v(x)|1_{W_i}(x) \inf_{z\in W_i} \Big((\delta(x)+t)\tilde{N}(\nabla u)(\tilde{z})+u(\tilde{z})\Big).\]
This function is integrable, since by the Cacciopolli inequality and Boundary H\"{o}lder (\refprop{prop:BoundaryHoelder}),
\begin{align*}
&\int_{\supp(\chi_k)} \sum_{i:W_i\cap\mathrm{supp}(\chi_k)\neq \emptyset}C|\nabla v(x)|1_{W_i}(x) \inf_{z\in W_i} \Big((\delta(x)+t)\tilde{N}(\nabla u)(\tilde{z})+u(\tilde{z})\Big) dx
\\
&\lesssim \sum_{i:W_i\cap\mathrm{supp}(\chi_k)\neq \emptyset}\Big(\int_{W_i}|\nabla v|^2\Big)^{1/2} \Big(\int_{W_i}\inf_{z\in W_i} \Big((\delta(x)+t)\tilde{N}(\nabla u)(\tilde{z})+u(\tilde{z})\Big)^2dx\Big)^{1/2}
\\
&\lesssim \sum_{i:W_i\cap\mathrm{supp}(\chi_k)\neq \emptyset}\Big(\int_{\frac{3}{2}W_i}\frac{|v-(v)_{W_i}|^2}{\delta^2}\Big)^{1/2} l(W_i)^{-n/2+1}\Big(\int_{\tilde{W}_i}((\delta(x)+t)\tilde{N}(\nabla u)+u) \, d\sigma\Big)
\\
&\lesssim \sum_{i:W_i\cap\mathrm{supp}(\chi_k)\neq \emptyset}l(W_i)^\beta\Big(\int_{\tilde{W}_i}((\delta(x)+t)\tilde{N}(\nabla u)+u) \, d\sigma\Big)
\\
&\lesssim \Vert\tilde{N}(\nabla u)\Vert_{L^1(\partial\Omega)} + \Vert u\Vert_{L^1(\partial\Omega)}<\infty.
\end{align*}
The last inequality follows by the Poincar\'{e} inequality from 
\[\Vert u \Vert_{L^1(\partial\Omega)}\lesssim \Vert \nabla u \Vert_{L^1(\mathcal{N})}\lesssim \Vert\tilde{N}(\nabla u)\Vert_{L^1(\partial\Omega)}.\] 
This finishes the proof of \eqref{eq:I}.
\medskip

\underline{\textbf{Step 2: Proof of \eqref{eq:II}:}} We split the integral into
\begin{align*}
    {\rm II}&=\int_\Omega (A\circ\tau_t^{(k)})\nabla (u\circ\tau_t^{(k)})\cdot \nabla(\chi_k\circ\tau^{(k)}_t) \eta_t v dx
    \\
    &\qquad + \int_\Omega (A\circ\tau_t^{(k)})\nabla (u\circ\tau_t^{(k)})\cdot \nabla\eta_t(\chi_k\circ\tau^{(k)}_t) v dx
    \\
    &\qquad + \int_\Omega\mathrm{div}((A\circ\tau_t^{(k)})\nabla(u\circ\tau_t^{(k)})) (\chi_k\circ \tau)\eta_t v dx=:{\rm II}_1+{\rm II}_2+{\rm II}_3.
\end{align*}
First, \({\rm II}_3=0\) by use of the PDE for \(u\). Next, we show that \({\rm II}_2\to 0\) as \(t\to 0\). The proof is similar to the proof for \({\rm I}_3\). We split
\begin{align*}
    |{{\rm II}}_2|\lesssim \int_{\supp(\nabla \eta_t)\setminus\Lambda^t} t^{-1}|\nabla (u\circ\tau_t^{(k)})| |v| dx + \int_{\supp(\nabla \eta_t)\cap\Lambda^t} \delta^{-1}|\nabla (u\circ\tau_t^{(k)})| |v| dx
    \\
    :={\rm II}_{21}+{\rm II}_{22}. 
\end{align*}
By Boundary H\"{o}lder (\refprop{prop:BoundaryHoelder}) like in \({\rm I}_{31}\), we have for the first term
\begin{align*}
    {{\rm II}}_{21}\lesssim t^\beta\int_{\partial\Omega}\tilde{N}(\nabla u) d\sigma\lesssim t^\beta \to 0\quad \textrm{ as }t\to 0.
\end{align*}
For the second integral, we follow the proof idea for \({\rm I}_{32}\). Using the same notation, we have by Boundary Hölder (\refprop{prop:BoundaryHoelder}), Cacciopolli, \eqref{eq:UpointwiseBoudnByN(nablau)3}, and \refass{ass:EPS0}
\begin{align*}
    |{\rm II}_{22}|&\lesssim\sum_j\int_{T(\Delta_{3t}(X_j))\cap\supp(\nabla \eta_t)}\delta^{-1}|v| |\nabla u\circ\tau_t^{(k)}|dy
    \\
    &\lesssim\sum_j\int_{T(\Delta_{3t}(X_j))\cap\supp(\nabla \eta_t)}\delta^{-1+\beta} \big(\fint_{\Delta_{c t}(\tilde{x}_j)}\tilde{N}(\nabla u) d\sigma\big)dx
    \\
    &\lesssim \sum_j\Big(\big(t^{-n+1}\int_{\Delta_{c t}(\tilde{x}_j)}\tilde{N}(\nabla u) d\sigma\big) \int_{T(\Delta_{3t}(X_j))}\delta^{-1+\beta}(y)dy\Big)
    \\
    &\lesssim t^\beta\int_{\partial\Omega}\tilde{N}(\nabla u) d\sigma \to 0 \quad \textrm{ as }t\to 0.
\end{align*}

Lastly, for \({\rm II}_1\), let \(\EPS>0\) be arbitrary, and set \(\Omega_\EPS=\Omega\setminus\{x\in \Omega;\delta(x)\leq \EPS\}\). We split
\begin{align*}{\rm II}_1&=\int_{\Omega_\EPS} (A\circ\tau_t^{(k)})\nabla (u\circ\tau_t^{(k)})\cdot \nabla(\chi_k\circ\tau^{(k)}_t) \eta_t v dx
\\
&\qquad + \int_{\{\delta(x)\leq \EPS\}} (A\circ\tau_t^{(k)})\nabla (u\circ\tau_t^{(k)})\cdot \nabla(\chi_k\circ\tau^{(k)}_t) \eta_t v dx={\rm II}_{11}^{(\EPS)}+{\rm II}_{12}^{(\EPS)}.
\end{align*}
By continuity of the \(L^2\) inner product, we have
\begin{align*}
{{\rm II}}_{11}^{(\EPS)}\to \int_{\Omega_\EPS} A\nabla u\cdot \nabla \chi_k v dx \quad \textrm{ when }t\to 0.
\end{align*}
For the other term, we obtain by \eqref{eq:UpointwiseBoudnByN(nablau)2}, the finite overlap of the Whitney cubes, and \(v\in L^\infty\) that
\begin{align*}
    {\rm II}_{12}^{(\EPS)}&\lesssim \sum_{i:W_i\cap\{\delta(x)\leq \EPS\}\cap\supp(\chi_k)\neq \emptyset} l(W_i)^n\Big(\fint_{W_i}|\nabla (u\circ\tau_t^{(k)})|^2dx\Big)^{1/2}
    \\
    &\lesssim \sum_{i:W_i\cap\{\delta(x)\leq \EPS\}\cap\supp(\chi_k)\neq \emptyset} l(W_i)\int_{\tilde{W}_i}\tilde{N}(\nabla u) d\sigma
    \lesssim \EPS \Vert\tilde{N}(\nabla u)\Vert_{L^1(\partial\Omega)}.
\end{align*}
Thus \({\rm II}_{12}^{(\EPS)}\) is uniformly bounded in \(t\) and \({\rm II}_{12}^{(\EPS)}\to 0\) when \(\EPS\to 0\). Hence 
$$
\lim_{t\to0}{{\rm II}}_1=\lim_{\EPS\to 0}\lim_{t\to0}({{\rm II}}_{11}^{(\EPS)} +{{\rm II}}_{12}^{(\EPS)})=\int_\Omega A\nabla u\cdot\nabla\chi_k v dx
$$
completes the proof of \eqref{eq:II}.

\medskip

\underline{\textbf{Step 3: Proof of \eqref{eq:III}:}} From the assumption of vanishing Neumann data for \(u\) in \eqref{eq:MixedPDEZeroData} we have that
\begin{align*}
\int_\mathcal{N}\fint_h^{2h}A\circ\tau_t^{(k)}\nabla (u\circ\tau_t^{(k)})\cdot \nu \,(\chi_k\circ\tau^{(k)}_t)\eta_t \Phi\,dt d\sigma\to 0\quad \textrm{ for }h\to 0, \Phi\in C^\infty(\overline{\Omega}).
\end{align*}
Since \(v\in C^\alpha(\overline{\Omega})\subset C(\overline{\Omega})\), we can approximate \(v\) by \(\Phi\in C^\infty(\overline{\Omega})\) in the \(L^\infty\)-norm arbitrarily close. If we combine this with
\begin{align*}
&\int_\mathcal{N}\fint_h^{2h}A\circ\tau_t^{(k)}\nabla (u\circ\tau_t^{(k)})\cdot \nu \,(\chi_k\circ\tau^{(k)}_t)\eta_t (v-\Phi)\,dt 
\\
&\lesssim \Vert v-\Phi\Vert_{L^\infty(\Omega)}\int_{\mathcal{N}\cap\supp(\chi_k)}\Big(\fint_{B_{h/2}(\tilde{x},\psi_k(\tilde{x})+3h/2)}|(A\circ\tau_t^k)\nabla (u\circ\tau_t^k)\cdot \nu|dy\Big) d\sigma(\tilde{x})
\\
&\lesssim \Vert v-\Phi\Vert_{L^\infty(\Omega)}\Vert\tilde{N}(\nabla u)\Vert_{L^1(\partial\Omega)}\lesssim \Vert v-\Phi\Vert_{L^\infty(\Omega)},
\end{align*}
we can also conclude that 
\begin{align*}
    {\rm III}&=\int_{\mathcal{N}} \fint_{h}^{2h}(A\circ\tau_t^k)\nabla (u\circ\tau_t^k)\cdot \nu \,(\chi_k\circ\tau_t^k)\eta_t v \,dtd\sigma
    \to 0\quad\textrm{ for }h\to 0.
\end{align*}

\begin{rem}
    In the proof of the uniqueness, \refthm{MainThm} \eqref{item:MainThmB}, we had to use the DKP condition (\refass{ass:largeDKP}) only to bound \({\rm I}_1\). In fact, the DPR condition (\refass{ass:largeDPR}) is not enough here and this shows in equation \eqref{eq:UniquenessDKPneededInsteadOfDPR}: Using the DKP condition, we can replace \(|A-A\circ\tau_t^{(k)}|\) with \(t|\nabla A|\) by the fundamental theorem of calculus, where it is important that the factor \(t\) goes to zeros depending only on \(t\) and not on the Whitney cube. For instance, this allows us to conclude that \(A\circ\tau^{(k)}\) converges to \(A\) pointwise almost everywhere, whereas the DPR condition does not give this conclusion.
\end{rem}

\begin{rem}
    \refass{ass:EPS0} is relevant for the bound for \({\rm I}_{32}\) and \({\rm II}_{22}\). For both we require \(\EPS_0\leq\beta\), where \(\beta\) is the exponent of the Boundary H\"{o}lder inequality (\refprop{prop:BoundaryHoelder}). Since \(\beta>0\) depends only on the Lipschitz constant \(l\), the ellipticity constant \(\lambda\), and the dimension \(n\), \(\EPS_0\) only depends these parameters.
\end{rem}

\bibliographystyle{alpha}
\bibliography{references} 
\end{document}